\documentclass[12pt,a4paper,twoside,notitlepage]{article}
\usepackage[english]{babel}
\usepackage{graphicx}
\usepackage{amsmath,amsthm,amsbsy,epsfig,amsfonts}  
\def\resume{\if@twocolumn
\section*{R\'esum\'e}
\else \small 
\quotation{\bf \it R\'esum\'e \rule[1mm]{1.5mm}{0.2mm}\vspace{0pt}} 
\fi}
\def\endresume{\if@twocolumn\else\endquotation\fi}

\newcommand{\pequationdeb}{$$ \left\{ \begin{minipage}[c]{148mm}}
\newcommand{\pequationfin}{\end{minipage}
                           \right. $$}
\newcommand{\beq}     {\begin{equation}}
\newcommand{\enq}     {\end{equation}}
\newcommand{\DP}[2]{\partial_{#2} #1}
\newcommand{\DPO}[1]{\displaystyle \frac{{\cal D}_a}{{\cal D} t}#1}
\newcommand{\DT}[1]{\displaystyle \frac{{\rm d}}{{\rm d}\,t}#1}

\newcommand{\Int}{\displaystyle \int}
\newcommand{\vite}{{\bf u}}
\newcommand{\Norm}{{\bf N}}
\newcommand{\oeta}{\overline{\boldsymbol{\eta}}}
\newcommand{\oxi}{\overline{\xi}}
\newcommand{\tang}{{\bf T}}
\newcommand{\fsec}{{\bf f}}
\newcommand{\nablah}{{\boldsymbol \nabla}}
\newcommand{\extra}{\boldsymbol \tau}
\newcommand{\extras}{\boldsymbol \sigma}
\newcommand{\symet}{{\bf D}}
\newcommand{\gradu}{{\boldsymbol \nabla {\bf u}}}
\newcommand{\gradv}{{\boldsymbol \nabla {\bf v}}}
\newcommand{\Mga}{{\bf m}}

\newcommand{\We}{\mbox{We}}
\newcommand{\Frac}[2] {\frac{\mbox{\normalsize{$#1$}}}{\mbox{\normalsize{$#2$}}}}
\def\rit{\hbox{\it I\hskip -2pt R}}
\newtheorem{theorem}{Theorem}[section]
\newtheorem{lemma}[theorem]{Lemma}

\newtheorem{remark}[theorem]{Remark}

\begin{document}
\title{Well-posedness of the equations of a viscoelastic fluid with a free boundary}
\author{Herv\'e VJ. Le Meur}
\date{October 20th 2009}
\maketitle
%
%
\begin{abstract}
{\it 
In this article, we prove the local well-posedness, for arbitrary
initial data with certain regularity assumptions,
of the equations of a Viscoelastic Fluid of
Johnson-Segalman type with a free surface. More general constitutive
laws can be easily managed in the same way. The geometry is defined by
a solid fixed bottom and an upper free boundary submitted to surface
tension.  The proof relies on a Lagrangian formulation. First we
solve two intermediate problems through a fixed point using mainly
\cite{G.AllainAMO1987} for the Navier-Stokes part. Then we solve
the whole Lagrangian problem on $[0,T_0]$ for $T_0$ small enough through
a contraction mapping. Since the Lagrangian solution is smooth, we can
come back to an Eulerian one.}
\end{abstract}

\section{Introduction}

This article deals with the equations modeling the flow of a
viscoelastic fluid with a free surface. The fluid is assumed to be
incompressible, viscous and its stress tensor contains both a viscous
and a viscoelastic part. For the proof, the latter obeys a
Johnson-Segalman constitutive law and more general models are studied
in an appendix. The geometry is 2D, horizontally infinite and vertically
bounded by a rigid bottom and a free boundary.

Gravity and surface tension are the only external forces.\\

Unless specified, all the articles mentionned hereafter consider
viscous fluids obeying the Navier-Stokes equations, but not
viscoelastic fluids obeying more complex laws.

A similar problem in a {\em bounded} geometry (drop of a fluid) was
dealt with in numerous articles by Solonnikov. He wrote an article
\cite{Solonnikov_77} in 1977 in which he proved the local in time
unique solvability in the H\"older space ${\cal
 C}^{2+\alpha,1+\alpha/2}$ ($1/2 < \alpha <1$) (no surface
tension). He also proved the global existence with no source term and
sufficiently small initial data in \cite{Solonnikov_86} in the space
$W^{2,1}_p$ with $p>n$ (no surface tension). More recently Shibata and
Shimizu \cite{Shibata_Shimizu_07} improved this result in
$W^{2,1}_{q,p}$ ($=L^p(0,T;W^{2,q}(\Omega))\bigcap
W^{1,p}(0,T;L^q(\Omega))$) and still with no surface tension. The latter
article also contains an interesting review and related problems. Surface tension was
included in other articles by Solonnikov where he proved local
existence in time and uniqueness (in $W^{2+\alpha,1+\alpha/2}_2$ with
$1/2 < \alpha <1$ for any initial data in \cite{Solonnikov_84}) and
global existence for initial data sufficiently close to equilibrium in
\cite{Solonnikov_86}.

An other direction of research is the flow of a fluid down an inclined
plane on which Teramoto proved the local in time unique solvability in
3D without surface tension in \cite{Teramoto_85} and with surface
tension in \cite{Teramoto_92}. Nishida, Teramoto and Win
\cite{Nishida_Teramoto_Win} proved the global in time unique existence
in a periodic 2D domain and with sufficiently small initial data. In
\cite{Bresch_Noble_07}, Bresch and Noble derived a shallow water model
and obtained estimates $\varepsilon \rightarrow 0$ (still for a
periodic in $x$ flow).\\

Here, we investigate an other direction in which we must quote the
pioneer article of Beale in 1981 who proved existence and uniqueness
in small time in 3D without surface tension in \cite{JTBeale81}, using
Sobolev-Slobodetski\v{\i} spaces. He also proved in \cite{Beale_84}
the global existence for sufficiently small initial data, thanks to
gravity {\em and} surface tension. Fujita-Yashima proved in 1985 the
existence of a stationnary and a time-periodic solution in the same
geometry with surface tension in \cite{Fujita-Yashima_85} by
perturbation methods. Allain adapted the articles of Beale to the 2D
geometry with surface tension. She proved well-posedness in
\cite{G.AllainAMO1987} and \cite{G.Allaintoulouse}. In
\cite{Sylvester_90}, Sylvester had large time existence even without
surface tension. Tani extended these results to the 3D case in
\cite{Tani_96}. In a joint article with Tanaka, he gave a proof of
large time existence for sufficiently small initial data whether there
is surface tension or not in \cite{Tani_Tanaka_95}. In
\cite{Abels_05}, Abels generalized these results to the $L^q$ spaces
($N < q < \infty$) in 2 or 3 dimensions (the velocity is in
$W^{2,1}_q$). In \cite{Tanaka_Tani_03}, Tanaka and Tani proved in 2003
the local existence for arbitrary initial data and global existence
for sufficiently small initial data of a compressible Navier-Stokes
flow with heat and surface tension taken into account.

Recent articles were published on the numerical resolution of the
latter problem especially in view of the simulation of surface
waves. We may quote Guti\'errez and Bermejo
\cite{Gutierrez_Bermejo_05}, Audusse {et al.}
\cite{Audusse_Bristeau_Decoene_06}, Guidorzi and Padula
\cite{Guidorzi_Padula_07} and in 2009 Fang {et al.}
\cite{Fang_Parriaux_Rentschler_Ancey}.  \\

Our main result is the well-posedness of the equations of a
viscoelastic fluid with a free boundary for arbitrary large initial
data sufficiently regular. Most of the results of the present article
are announced in \cite{LeMeur_95} and extend those of the author's PhD
thesis \cite{hervethese?}.

Up to some minor modifications, more complex constitutive laws can be
dealt with. Basically, it relies on the fact that the constitutive law
improves the regularity of its source term which is the velocity
gradient. The nonlinear terms are easy to handle because of an
algebra property. So, provided one may have estimates of the
extra stress in an algebra, the same theorem applies to more general
viscoelastic fluids.

Below, we give the Eulerian (Subsection 2.1), and then Lagrangian
equations (Subsection 2.2). The Equations given, we set the spaces and
operators to be inverted and give the sketch of the proof 
(Subsection 2.3). Section 3 is devoted to solving an auxiliary
problem which proof is sketched in Subsection 3.1. Then, one
solves a second nonlinear auxiliary problem specific to the
viscoelastic fluids in Section 4. Section 5 is devoted to estimates of
the error terms and we use all the preceding results in a fixed point
in Section 6.

\section{The equations and the sketch of the proof}

Starting from the dimensionless equations in Eulerian coordinates
(Subsection 2.1), we derive the dimensionless equations in Lagrangian
coordinates (Subsection 2.2), state the operators, spaces and give the
sketch of the proof (Subsection 2.3).

\subsection{Eulerian equations}

The dimensioned variables and fields are tilded. After 
exhibiting the dimensioned system of equations, we make it
dimensionless.

The domain of the flow is denoted $\Omega(t) \subset
\mathbb{R}^2$. Its bottom $S_B$ is given independent of time, and
represented by a depth function $b(\tilde{x}_1)$ such that at the
bottom $\tilde{z}=-b$. Initially $\Omega(\tilde{t}=0)$ is denoted
$\Omega$. The upper boundary is a free surface $S_F(\tilde{t})$ at
time $\tilde{t}$. At $\tilde{t}=0$, it is denoted $S_F$ and
represented by a height function $\zeta$. We assume that these surfaces do
not cross (the bottom does not dry even at infinity).  A typical
domain can be seen on Figure 
\ref{fig1}.

\begin{figure}[htbp]
\begin{center}
\includegraphics[width=7cm]{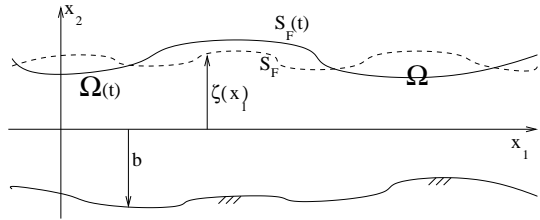}
\end{center}
\caption{Domain of the flow.}
\label{fig1}
\end{figure}

The functions $b$ and $\zeta$ are the initial $H^{5/2+r}$ ($0 < r < 1/2$)
height
functions of the bottom $S_B$ and of the free surface $S_F$
respectively. We assume that $\zeta$ tends to zero at $\pm \infty$ and
$b$ to some limit at $\pm \infty$. So the domain is unbounded but of
finite depth in the vertical direction. If we denote
$\tilde{x}=(\tilde{x}_1,\tilde{x}_2)$ a current point in $\Omega$ (so
at $\tilde{t}=0$), then $\Omega = \{\tilde{x}/ b(\tilde{x}_1) < \tilde{x}_2 <
\zeta(\tilde{x}_1) \}$.

Hereafter, vectors and tensors are written in bold letters. Their
components are in non-bold letters and with corresponding indices. We
use the summation convention and indices after a comma designate a
differentiation with respect to the variable:
$\partial_{\tilde{x}_j}\tilde{u}_i=\tilde{u}_{i,j}$. We denote
$\tilde{{\bf v}}$ the velocity, $\tilde{p}$ the pressure, and
$\tilde{\extra}$ the extra stress tensor due to the viscoelasticity in
the full system:

\begin{equation}
\label{eq2}
\left\{
\begin{array}{lll}
\rho\left( \DP{\tilde{{\bf v}}}{\tilde{t}} +\tilde{{\bf v}} . \tilde{\nabla} \tilde{{\bf v}} \right)-\mu_{sol} \tilde{{\boldsymbol \Delta}} \tilde{{\bf v}} +\tilde{{\boldsymbol \nabla}} \tilde{p}& = \tilde{\mbox{div}} \: \tilde{\extra} -\rho \tilde{g}_0 \vec{\mbox{\j}} & \mbox{ in } \Omega(t) \times (0,T)\\
\tilde{\mbox{div}} \: \tilde{{\bf v}} & = 0 & \mbox{ in } \Omega(t) \times (0,T),\\
\tilde{\extra} +\lambda \displaystyle \frac{{\cal D}_a[\tilde{{\bf v}}] }{{\cal D} \tilde{t}}\tilde{\extra}  & = 2\mu_{pol} \tilde{\symet}[\tilde{{\bf v}}]& \mbox{ in } \Omega(t) \times (0,T),\\
\tilde{\extra} . {\bf n}-\tilde{p} {\bf n}+2\mu_{sol}\tilde{\symet}[\tilde{{\bf v}}].{\bf n}-\tilde{\alpha}\tilde{H} {\bf n} & = -P_{atm}{\bf n} & \mbox{ on } S_F(t) \times (0,T),\\
\tilde{{\bf v}}& =0 &\mbox{ on } S_B,\\
\tilde{{\bf v}}(\tilde{x},\tilde{t}=0)& =\tilde{\vite}_0(\tilde{x}) & \mbox{ in } \Omega,\\
\tilde{\extra}(\tilde{x},\tilde{t}=0) &=\tilde{\extra}_0(\tilde{x})& \mbox{ in } \Omega.
\end{array}
\right.
\end{equation}

In this system, $\rho$ is the density of the fluid, $\mu_{sol}$ the
solvent viscosity, $\tilde{g}_0$ is the acceleration of gravity,
$\lambda$ the relaxation time, $\mu_{pol}$ the polymeric viscosity,
$\tilde{\alpha}$ the surface tension coefficient, $P_{atm}$ the atmospheric
pressure, $\symet[\tilde{{\bf v}}]$ the symmetric part of the
velocity gradient (rate of strain tensor) and

\begin{equation}
\label{eq3}
\begin{array}{l}
\displaystyle \frac{{\cal D}_a[\tilde{{\bf v}}] }{{\cal D} \tilde{t}}\tilde{\extra} = \DP{\tilde{\extra}}{\tilde{t}}+\tilde{{\bf v}} .\tilde{{\boldsymbol \nabla}} \tilde{\extra} -g_a(\tilde{\nablah} \tilde{{\bf v}}, \tilde{\extra})\\
\mbox{ where }g_a(\gradv, \extra)=\frac{a-1}{2}\left(\gradv ^T\extra +\extra \, \gradv \right)+\frac{a+1}{2}\left(\extra \, \gradv^T+\gradv \, \extra\right),
\end{array}
\end{equation}
is the interpolated ($a \in [-1,1]$) Johnson-Segalman time derivative
of tensors designed to let the tensors remain frame-invariant. The
viscoelastic fluid is supposed to have its total stress tensor equal
to the sum of a diagonal pressure matrix, a viscous term and an
extra stress tensor $\tilde{\extra}$. The constitutive equation
(\ref{eq2})$_3$ satisfied by $\tilde{\extra}$ describes the behavior
of complex fluids that have a memory (see \cite{JosephFDVEL}). Here we
write the Oldroyd B model but other models could be treated the
same way.

To get dimensionless equations, we define a characteristic
length $L$ (uniform in every dimension) and a characteristic velocity
$U_0$. They enable to define new dimensionless variables untilded:
\begin{equation}
\label{eq4}
\tilde{x}=(\tilde{x}_1,\tilde{x}_2)=L x =L(x_1,x_2), \; \; \tilde{t}= \frac{L}{U_0} \: t,
\end{equation}
new dimensionless fields untilded:
\begin{equation}
\label{eq5}
\begin{array}{c}
\tilde{{\bf v}}(\tilde{x},\tilde{t})= U_0 {\bf v}(x,t),\: \tilde{p}(\tilde{x},\tilde{t})=P_{atm}-\rho \tilde{g}_0 L x_2+(\mu_{sol}+\mu_{pol})\Frac{U_0}{L} p(x,t), \\
\tilde{\extra}(\tilde{x},\tilde{t})=\Frac{(\mu_{sol}+\mu_{pol})U_0}{L}\extra(x,t), 
\end{array}
\end{equation}
and some dimensionless numbers:
\begin{equation}
\label{eq6}
\mbox{Re} =\Frac{\rho L U_0}{\mu_{sol}+\mu_{pol}}, \: \We = \Frac{\lambda U_0}{L}, \: \varepsilon=\Frac{\mu_{pol}}{\mu_{sol}+\mu_{pol}}, \: \alpha=\Frac{\tilde{\alpha}}{U_0(\mu_{sol}+\mu_{pol})}, \: g_0=\Frac{\rho \tilde{g}_0\, L^2}{(\mu_{sol}+\mu_{pol})U_0}.
\end{equation}
So we denote Re the Reynold's number, We the Weissenberg number, and
$\alpha$ the dimensionless surface tension. All these changes enable
to make (\ref{eq2}) dimensionless:
\begin{equation}
\label{eq7}
\left\{
\begin{array}{llr}
\mbox{Re}\left( \DP{{\bf v}}{t}+{\bf v} .\nablah \,{\bf v}\right)-(1-\varepsilon)\Delta {\bf v}+\nablah p - \mbox{div } \extra &=0 & \mbox{ in }\Omega(t) \times (0,T) \\
\mbox{div } {\bf v} &=0 & \mbox{ in }\Omega(t) \times (0,T) \\
\extra+\We \DPO{[{\bf v}]{\extra}} -2\varepsilon\symet[{\bf v}]& =0 &  \mbox{ in } (\Omega(t) \times(0,T)),\\
-p {\bf n} +2(1-\varepsilon) \symet[{\bf v}] \cdot {\bf n} +\extra \cdot {\bf n} -\alpha H {\bf n} +g_0 x_2 \:{\bf n}&=0 &  \mbox{ on }S_F(t) \times (0,T),\\
{\bf v}(x,t)&=0 &  \mbox { on }S_B,\\
{\bf v}(x,0)&=\vite_0(x) & \mbox{ in } \Omega,\\
\extra(x,0)&=\extras_0(x) & \mbox{ in } \Omega.
\end{array}
\right.
\end{equation}
This system of partial differential equations is supposed to describe,
in the Eulerian coordinates, the flow of a viscoelastic fluid submited
to surface tension and gravity. We have explicitely used that $g_a$
is bilinear (see (\ref{eq3})), but more complex constitutive laws
without any higher order derivative of $\vite$ and of $\extra$ can be
included in the proof with minor changes. This will be depicted in Appendix A.

\subsection{Lagrangian equations}
We need to define the function 
\begin{equation}
\label{eq8}
 \begin{array}{rrcl}
\oeta(.,t) :& \Omega & \rightarrow & \Omega(t) \\
& X & \mapsto & \oeta(X,t),
\end{array}
\end{equation}
which gives the location at time $t$ of the point that used to be at
$X \in \Omega$ at time $t=0$. In small times, $\oeta$ will
be close to identity and we write 
\begin{equation}
\label{eq8.5}
\oeta(X,t)=X+\eta(X,t).
\end{equation}
In a sense that will be made clearer later, the displacement
$\eta$ is ``small''.
Let us then define the fields in the
Lagrangian coordinates:
\begin{equation}
\label{eq9}
\vite(X,t)={\bf v}(\oeta(X,t),t) ; \; q(X,t)=p(\oeta(X,t),t); \; \extras(X,t)=\extra(\oeta(X,t),t),
\end{equation}
where these fields are respectively the velocity, the pressure
and the extra stress tensor (modeling the polymer). We need also
to define some geometrical quantities:
\begin{equation}
\label{eq10}
\begin{array}{c}
({\rm d}\oeta)_{ij}=\DP{\oeta_i}{X_j}(X,t)=\oeta_{i,j}(X,t);  \, ({\boldsymbol {\overline \xi}})=({\rm \bf d} \oeta)^{-1}(X,t);\\
 \Norm(X,t)= (-\zeta'(X_1),1)/\sqrt{1+\zeta'^2}; \,{\boldsymbol{\mathcal{N}}}=(N_1-\partial_{\tang}\eta_2, N_2+\partial_{\tang}\eta_1),
\end{array}
\end{equation}
where $\partial_{\tang}= (1+\zeta'^2)^{-1/2} \partial_{X_1}$ is the
tangential derivative on $S_F$ (applied to functions that depend only
on $X_1$), $\Norm$ is the unit normal to $S_F$ pointing upward, and
$\boldsymbol{\mathcal{N}}$ is a vector that will appear later.

In (\ref{eq7}) we must transform the terms to the Lagrangian
coordinates on $S_F$. The most difficult term is $\alpha H(x,t){\bf n}(x,t)$ which is
defined on $x \in S_F(t)$. We will make explicit this transformation and let
the reader check the other terms.

Simple differential geometry gives us the unit vector $\Norm$
normal to $S_F$. In a
similar vein, the coordinates of a current point on $S_F(t)$ are
$(X_1+\eta_1(X_1,\zeta(X_1),t),\zeta(X_1)+\eta_2(X_1,\zeta(_1),t))$. If we denote
\begin{equation}
\label{eq11}
\begin{array}{rcl}
\eta_{1,X_1} & = & \partial_{X_1}(\eta_1(X_1,\zeta(X_1),t)),\\
\eta_{2,X_1} & = & \partial_{X_1}(\eta_2(X_1,\zeta(X_1),t)),
\end{array}
\end{equation}
a unit tangent vector to $S_F(t)$ is 
\[
\tang= (1+\eta_{1,X_1}, \zeta'(X_1)+\eta_{2,X_1})/\sqrt{(1+\eta_{1,X_1})^2+(\zeta'(X_1)+\eta_{2,X_1})^2}.
\]
>From the Frenet-Serret formula, one may write:
\[
H{\bf n} = \Frac{1}{\sqrt{(1+\eta_{1,X_1})^2+(\zeta'(X_1)+\eta_{2,X_1})^2}}\Frac{{\rm d}\tang}{{\rm d} X_1}= \Frac{\sqrt{1+\zeta'^2}}{\sqrt{(1+\eta_{1,X_1})^2+(\zeta'(X_1)+\eta_{2,X_1})^2}} \partial_{\tang}\tang,
\]
where $\partial_{\tang}(.)$ is the tangential derivative
along $S_F$. Since it is applied on functions only of $X_1$ we
have $\partial_{\tang}(.)=\partial_{X_1}(.) / \sqrt{1+\zeta'^2}$
and this derivative is more convenient since it is geometrical.
We may simplify further this formula by defining:
\begin{equation}
\label{eq11.5}
\Phi(X_1,t)=\Frac{\zeta'(X_1)+\eta_{2,X_1}(X_1,t)}{1+\eta_{1,X_1}(X_1,t)}-\zeta'(X_1)
\end{equation}
as is done in \cite{G.AllainAMO1987} (this article refers
to \cite{G.Allaintoulouse} where the reader will find some
further details). This definition enables to write
\begin{equation}
\label{eq12}
H{\bf n}(X_1,t)=\Frac{\sqrt{1+\zeta'^2}}{\sqrt{(1+\eta_{1,X_1})^2+(\zeta'(X_1)+\eta_{2,X_1})^2}}\partial_{\tang} \left( \left( \begin{array}{c}1 \\ \Phi(X_1,t)+\zeta'\end{array}\right)/\sqrt{1+(\Phi+\zeta')^2}\right).
\end{equation}
In the Lagrangian coordinates, we will need also an evolution equation
for $\Phi(X_1,t)$. Since 
\[
\DP{\oeta}{t}(X,t)=\vite(X,t),
\]
one finds easily the time derivative of $\Phi$:
\[
\Phi_t(X,t)=\Frac{\partial_{\tang}(\vite(X_1,\zeta(X_1),t)). \mathcal{N}(X_1,t)}{\mathcal{N}_2^2(X_1,t)}.
\]
Easy computations enable to complete the derivation of the
Lagrangian system of equations that may also be found in
classical books (\cite{JosephFDVEL}, \cite{Renardy_00}, ...):
\begin{equation}
\label{eq14}
\left\{
\begin{array}{llr}
\mbox{Re}\: u_{i,t}-(1-\varepsilon){\oxi}_{kj}({\oxi}_{lj} u_{i,l})_{,k}+{\oxi}_{ki}q_{,k}-\sigma_{ij,k}{\oxi}_{kj} & =0 & \mbox{in } \Omega\times(0,T),\\
{\oxi}_{kj} u_{j,k} & =0 & \mbox{in } \Omega\times(0,T),  \\
\sigma_{ij}+\We\left(\DP{\sigma_{ij}}{t}-\Frac{a-1}{2}({\overline {\xi}}_{li}u_{k,l}\sigma_{kj}+\sigma_{ik}u_{k,l}{\overline {\xi}}_{lj})\right.& & \\
\hspace*{2cm} \left. -\Frac{a+1}{2}(\sigma_{ik}{\overline {\xi}}_{lk}u_{j,l}+u_{i,l}{\overline {\xi}}_{lk}\sigma_{kj})\rule[6mm]{0cm}{0cm}\right) & &  \\
\hspace*{2cm}  -\varepsilon(u_{i,k}{\overline {\xi}}_{kj}+u_{j,k}{\overline {\xi}}_{ki}) & =0 & \mbox{in } \Omega\times(0,T),\\
\sigma_{ij}\underline{\cal N}_j-q\underline{\cal N}_i+(1-\varepsilon)({\overline {\xi}}_{kj}u_{i,k}+{\overline {\xi}}_{ki}u_{j,k})\underline{\cal N}_j+& & \\
\hspace*{2cm}+g_0(\zeta(X_1)+\eta_2(X_1,t))\underline{\cal N}_i -     & &         \\
\hspace*{2cm}\alpha \left( \partial_{\tang}\left( (1+(\Phi+\zeta')^2)^{\frac{-1}{2}}\left(\begin{array}{c}
                    1 \\
             \Phi+\zeta'
           \end{array} \right)\right)\right)_i & =0 &  \mbox{ on } S_F\times(0,T),\\
\Phi_t-\Frac{(\partial_{\tang}\vite)\cdot {\boldsymbol{\mathcal N}}}{{\mathcal N}_2^2} & =0 & \mbox{ on } S_F\times(0,T),\\
\Phi(t=0) & =0 & \mbox{ on } S_F,\\
\vite(X,0) & =\vite_0(X) & \mbox{in } \Omega,\\
\extras(X,0) &=\extras_0(X) & \mbox{in } \Omega,\\
\vite & =0 & \mbox{on } S_B\times (0,T).
\end{array}
\right.
\end{equation}

Here, $\varepsilon \in [0,1[$ is the dimensionless polymeric
viscosity, $g_0$ the acceleration of gravity, $\bf N$ the
outward unit normal to $S_F$, and
$\boldsymbol{\mathcal{N}}= ( N_1-\partial_{\tang} \eta_2,
N_2+\partial_{\tang} \eta_1)$ a (non-unit) vector normal to $S_F(t)$
at the point ${\oeta}(X_1,\zeta(X_1),t)$ convected back on $\Omega$,
defined on $(X_1,t)$ and geometrically defined with
$\partial_{\tang}$ instead of $\partial_{X_1}$.

In system (\ref{eq14}), there are equations inside $\Omega$
((\ref{eq14})$_1$ to (\ref{eq14})$_3$), equations at the free boundary
of $\Omega$ defined on $X_1 \in \mathbb{R}$ ((\ref{eq14})$_4$ and
(\ref{eq14})$_5$), initial conditions ((\ref{eq14})$_6$ to
(\ref{eq14})$_8$) and a Dirichlet condition on the bottom
(\ref{eq14})$_9$.

>From now on, the physical domain is the initial one and we consider
only the Lagrangian formulation.

\subsection{Operators, spaces and sketch of the proof}

\subsubsection{The operators}

Let us remember that for short time, ${\oeta}(X,t)=X+\eta(X,t)
\simeq X$. So the displacement $\eta$ will be small (in
small time) in a sense to be defined. For the same reason, if we define ${\boldsymbol \xi}$:
\[
\left( ({\rm d} \oeta)^{-1}=\right)\overline{\boldsymbol \xi}={\bf Id}+{\boldsymbol \xi},
\]

we see that ${\boldsymbol \xi}=({\bf Id}+{\rm d} {\boldsymbol \eta})^{-1}-{\bf
 Id}$ will be small.

If we define the operator $P({\boldsymbol \xi},\vite,q,\phi,\extras)$ as the
left-hand side of (\ref{eq14}) for equations (\ref{eq14})$_1$ to
(\ref{eq14})$_5$ and the initial conditions (\ref{eq14})$_7$ to
(\ref{eq14})$_8$, then (\ref{eq14}) amounts to solving
\[
P({\boldsymbol \xi},\vite,q,\phi,\extras)=(0,0,0,0,0,\vite_0,\extras_0),
\]
for $\vite$ vanishing on $S_B$ (see (\ref{eq14})$_9$) and $\Phi(t=0)=0$
(see (\ref{eq14})$_6$).

The orders of
magnitude of various terms must now be identified so as to see
(\ref{eq14}) as a perturbation of an invertible system.

In system (\ref{eq14}), there are some source terms : gravity and the
initial curvature that are not small even for small times. They are of
zeroth order and will be put in $P(0,0,0,0,0)$.

All the terms containing at least one ${\boldsymbol \xi}$ are stored
in $E$ (for ``Error''). The occurrence of a ${\boldsymbol \xi}$ will
enforce smallness.

Proceeding as in \cite{G.AllainAMO1987}, we collect the remaining terms in a new
operator acting on $(\vite,q,\phi,\extras)$ with the zeroth order in
${\boldsymbol \xi}$ (we put ${\boldsymbol \xi}$ in $E$). Moreover, we
keep only the linear in $\Phi$ operator from (\ref{eq14}) and put the
h.o.t. in $\Phi$ terms in $E^5$ (since $\Phi$ is initially vanishing these
nonlinear terms will be small in small time). We will denote
$P_1(\vite,q,\phi,\extras)$ this new operator with a minor change
$\phi=(1+\zeta'^2)^{-1}\Phi$. So an auxiliary problem to be solved is
$P_1(\vite,q,\phi,\extras)=(\fsec,a,\Mga,g,k,\vite_0,\extras_0)$ and
writes:

\begin{equation}
\label{eq15}
\left\{
\begin{array}{llr}
\mbox{Re}\DP{\vite}{t}-(1-\varepsilon)\Delta \vite +\nablah q -\mbox{div} \extras &= \fsec & \mbox{ in } \Omega_0\times(0,T),\\
\mbox{div} \; \vite &=a & \mbox{ in } \Omega_0\times(0,T),\\
\extras +\We \left(\DP{\extras}{t}-g_a(\gradu,\extras)\right) -2\varepsilon\symet[\vite]  &= \Mga & \mbox{in } \Omega_0\times(0,T), \\
\extras \cdot \Norm - q \Norm +2(1-\varepsilon)\symet[\vite]\cdot \Norm -\alpha\partial_{\tang}(\phi \Norm) & =g & \mbox{on } S_F\times (0,T),\\
\phi_t-(\partial_{\tang} \vite )\cdot \Norm & =k&  \mbox{on } S_F\times (0,T),\\
\vite(t=0) & =\vite_0(X) & \mbox{in } \Omega,\\
\extras(t=0) & =\extras_0(X) & \mbox{in } \Omega, \\
\end{array}
\right.
\end{equation}
for initially vanishing $\phi$, and $\vite=0$ on $S_B$. It is a
viscoelastic Stokes flow.  Notice that passing from (\ref{eq14}) to
(\ref{eq15}), the gravity term
$g_0(\zeta(X_1)+\eta_2){\boldsymbol{\mathcal{N}}}$ and the zeroth
order surface tension disappear because they are in the zeroth order
term $P(0,0,0,0,0)$. In the Navier-Stokes case $P_1$ is linear
(Stokes) but because of the nonlinearity of $g_a$, our operator $P_1$
is nonlinear.

With the preceding definitions, (\ref{eq14}) is equivalent to
\begin{equation}
\label{decompo_P}
\begin{array}{rcl}
P({\boldsymbol \xi},\vite,q,\phi,\extras) & = & P(0,0,0,0,0)+P_1(\vite,q,\phi,\extras)+E({\boldsymbol \xi},\vite,q,\phi,\extras)            \\
                & = & (0,0,0,0,0,\vite_0,\extras_0),
\end{array}
\end{equation}
for a velocity vanishing on the bottom and $\Phi(t=0)=\phi(t=0)=0$. In the
remaining, we will need also to define the nonlinear auxiliary
operator $P_2$ which gives the {\em nonlinear} behaviour close to
$(\vite_1,q_1,\phi_1,\extras_1)$:
\begin{equation}
\label{eq16}
\begin{array}{rl}
P_2[\vite_1,\extras_1](\vite,q,\phi,\extras):= &
P_1(\vite_1+\vite,q_1+q,\phi_1+\phi,\extras_1+\extras)-
P_1(\vite_1,q_1,\phi_1,\extras_1)\\
 = &\left(\begin{array}{l}
\mbox{Re}\, \DP{\vite}{t}-(1-\varepsilon)\Delta \vite +\nablah q -\mbox{div} \extras \\
\mbox{div} \vite \\
\extras +\We \left(\DP{\extras}{t}-g_a(\gradu,\extras) -g_a(\gradu_1,\extras)-\right. \\
\hspace*{3cm}\left. g_a(\gradu,\extras_1)\right)-2\varepsilon\symet[\vite] \\
\extras \cdot \Norm- q \Norm +2(1-\varepsilon)\symet[\vite]\cdot \Norm -\alpha\partial_{\tang}(\phi \Norm) \\
\phi_t-\partial_{\tang} \vite \cdot \Norm \\
\vite(t=0)\\
\extras(t=0)\\
\end{array}\right),
\end{array}
\end{equation}
for $\vite$ vanishing on $S_B$ and $\phi(t=0)=0$. In the Navier-Stokes
case $P_2 \equiv P_1$ because $P_1$ is linear.

\subsubsection{The functionnal spaces}

Following J.T. Beale \cite{JTBeale81} and G. Allain
\cite{G.AllainAMO1987}, we define anisotropic 
Sobolev-Slobodetski\v{\i} spaces for any $s$:
\[
\begin{array}{c}
 K^{s}(\Omega\times (0,T))=L^2(0,T;H^s(\Omega)) \bigcap H^{s/2}(0,T;L^2(\Omega)),            \\
 K^{s}(S_F\times (0,T))=L^2(0,T;H^s(S_F))\bigcap H^{s/2}(0,T;L^2(S_F)),
\end{array}
\]
where $H^s$ is the classical Sobolev space whose norm is denoted $\mid
.\mid_s$. We denote $\mid .\mid_{K^s}$ the norm
of the space $K^s$. The properties of these spaces, namely embedding and
trace theorems, can be found for instance in \cite{LionsMagenes}.

One may define the space of $(\vite,q,\phi,\extras)$ for $0< r <1/2$:
\begin{equation}
\label{eq17}
\begin{array}{rl}
  X_T^r(\Omega \times(0,T))= \{&(\vite,q,\phi,\extras) /  \\
& \vite \in K^{r+2} \mbox{ and } \vite=0 \mbox{ on }S_B\times(0,T);     \\
&    {\boldsymbol \nabla} q \in K^r \mbox{ and } q|_{S_F} \in H^{\frac{r}{2}+\frac{1}{4}}(0,T;L^2(S_F))  ;        \\
&    \partial_{\tang} \phi \in L^2(0,T;H^{r+\frac{1}{2}}(S_F)) \, ; \, \phi_t \in K^{r+\frac{1}{2}}(S_F\times(0,T)) ; \phi(0)=0     ;   \\
&   \extras \in H^{\frac{1+r}{2}}(0,T;H^{1+r}(\Omega))    \; \; \}.
\end{array}
\enq

This space $X_T^r$ is basically a regularity space with the condition
at the bottom for $\vite$ and the initial condition of $\phi$. The
$H^{\frac{1+r}{2}}(0,T;H^{1+r}(\Omega))$ regularity on $\extras$ is
needed to have an algebra so as to manage the nonlinear terms of a
wide variety of constitutive equations. 

Indeed, the $\gradu \, \extras$ terms of the constitutive law
must be estimated in $L^2(0,T;H^{1+r}) \bigcap
H^{\frac{1+r}{2}}(0,T;L^2)$. For $L^2(0,T;H^{1+r})$, if we use the
$L^2(0,T;H^{1+r})$ for $\gradu$, we need a $L^{\infty}(0,T;H^{1+r})$
estimate for $\extras$. For $H^{\frac{1+r}{2}}(0,T;L^2)$, if we use
the $H^{\frac{1+r}{2}}(0,T;L^2)$ estimate for $\gradu$, we need a
$H^{\frac{1+r}{2}}(0,T;L^{\infty})$ for $\extras$. So estimates of
$\extras$ in $H^{\frac{1+r}{2}}(0,T;H^{1+r})$ are natural.

The image space $Y_T^r$ is 
\begin{equation}
\label{eq18}
\begin{array}{rl}
 Y_T^r(\Omega\times(0,T))= \{ & (\fsec,a,\Mga,g,k,\vite_0,\extras_0)/  \\
   & \fsec \in K^r(\Omega \times (0,T)) ,           \\
   & a \in L^2(0,T;H^{r+1}(\Omega))\bigcap H^{\frac{r}{2}+1}(0,T,\, _0 \!H^{-1}(\Omega)), \\
   & \Mga \in K^{r+1}(\Omega \times (0,T)) , \\  
   & g,k \in K^{r+\frac{1}{2}}(S_F\times (0,T)) ,         \\
   & \vite_0,\extras_0 \in H^{r+1}(\Omega)\},
\end{array}
\end{equation}
where $_0\! H^{-1}(\Omega)$ is the dual space of $^0\! H^1=\{p\in H^1
/ p \equiv 0 \mbox{ sur } S_F\}$. We will need to keep the space
$H^{\frac{r}{2}+1}(0,T;\,_0\!H^{-1}(\Omega))$ so as to have a lift
field $\vite \in H^{\frac{r}{2}+1}(0,T,L^2(\Omega))$.

\subsubsection{Main result and sketch of the proof}

Our main result of well-posedness is the following:

\begin{theorem}
\label{theor2}
Let $0<r<1/2$, and the initial height functions of the free boundary $\zeta$
and of the bottom $b$ be such that $(-b-\displaystyle\lim_{\pm \infty}b) \in
H^{r+5/2}(\rit)$. Let also $(\vite_0,\extras_0)$ be in $H^{r+1}\times H^{r+1}_{sym}$ 
and the compatibility conditions 
\[
\begin{array}{c}
\mbox{div } \vite_0=0 \mbox{ in }\Omega,\\ 
\vite_0=0 \mbox{ on } S_B,\\
\extras_{0}.\Norm . {\tang}+(1-\varepsilon)(\symet[\vite_0]).\Norm . {\tang}-\alpha\partial_{\tang}\left( {\tang}\right).\Norm . {\tang}=0,
\end{array}
\]
where $\Norm$ is the unit normal and ${\tang}$ is the unit tangent
vector, be satisfied. Then there exists $T_0>0$ 
depending on the data ; $r, \Omega, \vite_0, 
\extras_0, \zeta, b, \We, \varepsilon, a$ and there exists a unique
$(\vite,q,\phi,\extras)\in X_{T_0}^r$ solution to the Lagrangian system (\ref{eq14}).
Under the same hypothesis, the Eulerian system (\ref{eq7}) admits
a solution with $\oeta \in H^1(0,T_0;H^{2+r}(\Omega))\bigcap
H^{2+\frac{r}{2}}(0,T_0;L^2(\Omega))$. The solution depends
continuously on the initial conditions provided they are in
a bounded subset of $(H^{r+1}(\Omega))^2$.
\end{theorem}

\begin{remark} One could even prove that the solution depends
continuously on the parameters $(\varepsilon, \mbox{Re}, \We, a)$ in
the domain $[0,1) \times \mathbb{R}^{+2} \times [-1,1]$. To that end,
we should track the occurrence of the parameters in the
constants. Since the constants in the estimates depend only
polynomialy on the parameters we can have the same estimates on any
parameter in an open neighborhood of the given parameter. This continuity is
very unlikely to be extendable to the whole domain at least since
$\varepsilon = 1$ would not have the regularizing effect of the
Navier-Stokes equation.
\end{remark}

In a first part (Section 3) we solve (\ref{eq15}) also written
$P_1(\vite,q,\phi,\extras)=(\fsec,a,\Mga,g,k,\vite_0,\extras_0)$.  For
that purpose, we first introduce a reduced auxilliary problem
$P_2[\vite_1,\extras_1](\vite,q,\phi,\extras) =(\fsec,0,\Mga,0,0,0,0)$
with zero initial conditions (lifted in the $\vite_1, \extras_1$) that
we solve through a fixed point between the Navier-Stokes part (solved
in \cite{G.AllainAMO1987}) and the constitutive equation. This step
needs uniform (Subsection \ref{sub33}) and contraction (Subsection
\ref{sub34}) estimates. Using this auxilliary problem, we solve
(\ref{eq15}) and so invert $P_1$ in Subsection \ref{sub35}. In
Subsection \ref{sub36} we show the continuous dependence with respect
to the right-hand side (rhs) of (\ref{eq15}) which comprises the initial
conditions. To show that the final mapping $P_1^{-1}$ is Lipschitz, we
have to check that this inverted operator is bounded (independently of
$T<T_0$).\\

In a second part (Section 4), for given $(\vite_1,\extras_1)$ in a
bounded ball of $K^{2+r}\times K^{r+1}\bigcap
L^{\infty}(0,T_0;H^{1+r})$, we prove there exists a unique solution in
$X_T^*$ ($=X_T^r$ with zero initial conditions) to the problem
$P_2[\vite_1,\extras_1](\vite,q,\phi,\extras)= (\fsec,a,\Mga,g,k,0,0)$
and prove boundedness of $P_2[\vite_1,\extras_1]^{-1}$ with a constant
independent of $(\vite_1,\extras_1)$ in a given ball.\\

In the third part of the proof (Section 5), we derive some estimates
on the ``error'' terms $E$ insuring that they are small and
contractant for $T_0$ small enough. \\

Last, in Section 6, we gather these results to build a functional $F$
involving $P_1, P_2$ and $E$ that is
contracting. The regularity found ensures then the existence and
uniqueness of a solution to the eulerian problem (\ref{eq14}). This
completes the proof.

\section{Inversion of the operator $P_1$}

In the present section, we prove the following Theorem of well-posedness:

\begin{theorem}
\label{thlineaire}
Let $0<r<1/2$, $B>0$ be given and the compatibility conditions:
\[
\begin{array}{c}
\mbox{div } \vite_0=0 \mbox{ in }\Omega,\\ 
\vite_0=0 \mbox{ on } S_B,\\
(\extras_{0\, i,j}).\Norm . {\tang}+(1-\varepsilon)(\symet[\vite_0]).\Norm . {\tang}-\alpha\partial_{\tang}\left( {\tang}\right).\Norm . {\tang}=0,
\end{array}
\]
be satisfied. There exists a time $T_0$
such that if $T\leq T_0$, for any $(\fsec,a,\Mga,g,k,\vite_0,\extras_0
)$ in the ball in $Y_T^r$ of radius $B$ denoted $\mbox{B}_{Y_T^r}(0,B)$, there
exists a unique $(\vite,p,\phi,\extras)\in X_T^r$ solution of
(\ref{eq15}). Moreover the (nonlinear) operator 
$(\vite,p,\phi,\extras)=P_1^{-1}(\fsec,a,\Mga,g,k,\vite_0,\extras_0 )$
is such that
\begin{equation}
\label{2ch6}
\begin{array}{l}
\mid \vite-\vite',p-p',\phi-\phi',\extras-\extras' \mid_{X_T^r} \leq \\
C \mid \fsec-\fsec',a-a',\Mga-\Mga',g-g',k-k',\vite_0-\vite'_0 , \extras_0-\extras'_0  \mid_{Y_T^r },
\end{array}
\end{equation}
where $C=C(\varepsilon,\We, B, T_0, \mbox{Re})$ is a constant
that does not depend on $T\leq T_0$.
\end{theorem}

\begin{proof}

To prove this Theorem, we need to solve a restricted problem
$P_2[\vite_1,\extras_1](\vite,q,\phi,\extras)= (\fsec,0,\Mga,0,0,0,0)$
with vanishing initial conditions in which $\vite_1$ and $\extras_1$
lift the initial conditions. As a consequence, notice that the continuity with respect to the initial conditions means continuity also with respect to those fields. Subsections 3.1 to 3.5 are devoted to
this, while Subsection \ref{sub35} proves existence of $P_1^{-1}$ and
Subsection \ref{sub36} its boundedness (\ref{2ch6}). Local continuity with respect to the initial conditions is a mere consequence of (\ref{2ch6}).

\subsection{A reduced second auxiliary problem}

Let $(\vite_1, \extras_1)$ be in $K^{r+2} \times K^{r+1}\bigcap
L^{\infty}(0,T_0;H^{1+r})$. To prove Theorem \ref{thlineaire}, we
need to solve
\begin{equation}
\label{2ch7}
P_2[\vite_1,\extras_1](\vite,q,\Phi,\extras)=(\fsec,0,\Mga,0,0,0,0).
\end{equation}
for $(\vite,q,\Phi,\extras)\in X_T^r$.

The resolution of the non-reduced problem
$P_2[\vite_1,\extras_1](\vite,q,\phi,\extras)=(\fsec,a,\Mga,g,k,0,0)$
is postponed to Section \ref{sect4} and will use the resolution of
(\ref{2ch7}).

To prove Theorem \ref{thlineaire}, in a first step, we solve
(\ref{2ch7}) in Subsections \ref{sub33} and \ref{sub34}. In Subsection
\ref{sub35}, we invert $P_1$ and in Subsection \ref{sub36}, we prove
that its inverse is bounded (\ref{2ch6}).

Let us begin by a very brief sketch of the proof, then a more precise
justification of the required estimates. We report the beginning of
the proof to Subsection \ref{sub33} after recalling some lemmas in
Subsection \ref{sub32}.\\

In the system (\ref{2ch7}), where $P_2$ is defined in (\ref{eq16}), we
will split the equations with the Stokes/Navier-Stokes part on the one
hand and the constitutive law on the other hand. This splitting is
done by numerous authors (see {\em e.g.} Guillop\'e-Saut \cite{GSNATMA90}
and Renardy \cite{Renardy_85}). Then, given $\vite^n, q^n,\phi^n$, we
will find an extra stress $\extras^{n+1}$. This extra stress may be
carried forward in the Stokes-like system solved by
\cite{G.AllainAMO1987} and will provide $\vite^{n+1}, q^{n+1},\phi^{n+1}$.

Then we will prove that this sequence $\vite^{n}, q^{n},\phi^{n},
\extras^{n}$ is a Cauchy sequence in a Banach space.\\

Let us be more precise below on the required estimates for the full
proof. Let $\vite^{0}=0, q^{0}=0,\phi^{0}=0, \extras^{0}=0$. We assume
we know $(\vite^n,q^n,\phi^n,\extras^n) \in X_{T_0}^r$ and we look for
$\extras^{n+1}$ such that:
\begin{equation}
\label{2ch16}
\left\{
\begin{array}{l}
\extras^{n+1} +\We\left(\DP{\extras^{n+1}}{t}-g_a(\gradu^n,\extras^{n+1})-g_a(\gradu_1,\extras^{n+1})-g_a(\gradu^n,\extras_1)\right) = 2\varepsilon\symet[\vite^n] +\Mga \\
\extras^{n+1}(t=0)=0.
\end{array}
\right.
\end{equation}

Such an equation is a simple ODE in the space $K^{1+r}$ with no loss of
regularity thanks to the Lagrangian simplification. Even, there is a
gain of regularity of $\extras^{n+1}$ with respect to $\gradu^{n}$
because there is no time derivative on $\gradu^n$ and no gradient on
$\extras^{n+1}$. The nonlinearity will be managed with uniform bounds
on $\mid \extras^n\mid_{H^{\frac{1+r}{2}}(0,T;H^{1+r})}$. Thus, the velocity
will occur only to the
source term $\gradu^{n}$. This enables to gain one level of time regularity and
prove estimates of $\extras^{n+1}$ in
$H^{\frac{1+r}{2}}(0,T;H^{1+r})$. Indeed, since $\gradu_1, \extras_1,
\gradu^n, \Mga$ are in $L^2(0,T_0; H^{1+r})$, we have $\extras^{n+1}
\in H^1(0,T_0;H^{1+r}) \subset H^{\frac{1+r}{2}}(0,T_0;H^{1+r})$. Notice that this
is true until time $T_0$.  This gain will enable to prove that
$(\vite^n, q^n, \phi^n, \extras^n) \mapsto (\vite^{n+1}, q^{n+1},
\phi^{n+1}, \extras^{n+1})$ is a contraction for $T_0$ sufficiently
small in $X_{T_0}^r$.

Then we look for $(\vite^{n+1},q^{n+1},\phi^{n+1},0)\in X_{T_0}^r$
solution of the Stokes like problem:

\begin{equation}
\label{2ch17}
\left\{
\begin{array}{lcl}
\mbox{Re}\, \DP{\vite^{n+1}}{t}-(1-\varepsilon)\Delta \vite^{n+1} +\nablah q^{n+1}  & = &\fsec +\mbox{div} \extras^{n+1}\; \mbox{ in } \Omega\times(0,T),\\
\mbox{div} \vite^{n+1} & = & 0\; \mbox{ in } \Omega\times(0,T),\\
- q^{n+1} \Norm +2(1-\varepsilon)\symet[\vite^{n+1}]\cdot \Norm -\alpha\partial_{\tang}(\phi^{n+1} \Norm) & =& -\extras^{n+1} \cdot \Norm  \; \mbox{ on } S_F\times (0,T),\\
\phi^{n+1}_t -(\partial_{\tang} \vite^{n+1} )\cdot \Norm & = & 0 \; \mbox{ on } S_F\times (0,T),\\
\phi^{n+1}(0) & = & 0 \; \;  \mbox{ on } S_F,\\
\vite^{n+1}(0)& = & 0 \; \; \mbox{ in } \Omega,\\
\vite^{n+1}   & = & 0 \; \; \mbox{ on } S_B \; \; \forall t.
\end{array}
\right.
\end{equation}

Theorem 4.1 of \cite{G.AllainAMO1987} in the case $\fsec=\fsec +\mbox{div} \extras^{n+1}, a=0,
g=-\extras^{n+1} \cdot \Norm, k=0, \vite_0=0$ states that for given
rhs and $T_0$ (even though it is not obvious in the way it is written,
one may state also this result thanks to the regularizing effect of
the Stokes operator), there exists a unique
$(\vite^{n+1},q^{n+1},\phi^{n+1})$ in $X_{T_0}^r$ if $\fsec
+\mbox{div} \extras^{n+1} \in K^r(\Omega \times (0,T_0))$ and
$\extras^{n+1} \cdot \Norm \in
K^{r+\frac{1}{2}}(S_F\times(0,T_0))$. Thanks to Theorem 4.1 of
\cite{G.AllainAMO1987}, we have

\[
\mid \vite^{n+1}, q^{n+1}, \phi^{n+1},0 \mid_{X_T^r}\leq C \mid \fsec + \mbox{div} \extras^{n+1} \mid_{K^r}+C\mid \extras^{n+1} \mid_{K^{r+\frac{1}{2}}(S_F \times (0,T))},
\]
where $C$ will denote a constant independent of $T\leq T_0$ until the
end of this article unless otherwise mentioned. From the assumptions and the definition of
$\extras^{n+1}$, we have that $\fsec+\mbox{div} \extras^{n+1} \in
K^r$.  Still thanks to the fact that $\extras^{n+1} \in H^1(0,T;H^{1+r})$ we know
that $\extras^{n+1} \cdot \Norm \in K^{r+1/2}(S_F \times
(0,T_0))$. So

\[
\mid \vite^{n+1}, q^{n+1}, \phi^{n+1},0 \mid_{X_T^r}\leq C \mid \fsec\mid_{K^r}+C\mid \extras^{n+1} \mid_{L^2(0,T;H^{1+r}) \bigcap H^{\frac{r}{2}}(0,T;H^{1}) \bigcap H^{\frac{r}{2}+\frac{1}{4}}(0,T;H^{\frac{1}{2}})}.
\]

So, for $(\vite^{n+1}, q^{n+1}, \phi^{n+1},0)$, we will need uniform
estimates of $\extras^{n+1}$ in various spaces.

For $(0,0,0,\extras^{n+1})$, we will even need to estimate
$\extras^{n+1}$ in $H^{\frac{1+r}{2}}(0,T;H^{1+r})$ which is even stronger.
So we will prove only the latter.\\

Concerning convergence, since the Stokes-like system is linear (it is not the
case of our constitutive equation !), we even have thanks to the
same Theorem 4.1 of \cite{G.AllainAMO1987}:
\[
\mid \vite^{n+1}-\vite^n, q^{n+1}-q^n, \phi^{n+1}-\phi^n,0 \mid_{X_T^r}\leq C \mid \mbox{div} (\extras^{n+1} -\extras^n) \mid_{K^r}+C \mid \extras^{n+1} -\extras^n \mid_{K^{r+\frac{1}{2}}(S_F \times 0,T)}.
\]

Thanks to the continuity, injection and trace theorems found in
\cite{LionsMagenes} (Chapter 4) or in \cite{JTBeale81} (Lemma 2.1):
\[
\begin{array}{c}
\mid \mbox{div} (\extras^{n+1} -\extras^n) \mid_{L^2(0,T;H^r)}\leq C \mid \extras^{n+1} -\extras^n\mid_{L^2(0,T;H^{1+r})}  ,                   \\
\mid \mbox{div} (\extras^{n+1} -\extras^n) \mid_{H^{\frac{r}{2}}(0,T;L^2)}\leq C \mid \extras^{n+1} -\extras^n\mid_{H^{\frac{r}{2}}(0,T;H^1)}, \\
\mid \extras^{n+1} -\extras^n \mid_{K^{r+\frac{1}{2}}(S_F \times (0,T))} \leq C \mid \extras^{n+1} -\extras^n \mid_{L^2(0,T;H^{1+r}(\Omega))\bigcap H^{\frac{r}{2}+\frac{1}{4}}(0,T;H^{\frac{1}{2}})}.
\end{array}
\]

So

\begin{equation}
\label{2ch17.5}
\mid \vite^{n+1}-\vite^n, q^{n+1}-q^n, \phi^{n+1}-\phi^n,0 \mid_{X_T^r}\leq C \mid \extras^{n+1} -\extras^n \mid_{K^{1+r} \bigcap H^{\frac{r}{2}}(0,T;H^{1}) \bigcap H^{\frac{r}{2}+\frac{1}{4}}(0,T;H^{\frac{1}{2}})}.
\end{equation}

To prove that the $(\vite^{n+1}, q^{n+1}, \phi^{n+1},0)$ part of
our sequence is of Cauchy type in a Banach space (and we will always
need $C$ independent of $T\leq T_0$), it suffices to prove thanks to
the interpolation Lemma \ref{GAcinqquatre}:

\[
\begin{array}{rl}
\mid \extras^{n+1} -\extras^n\mid_{L^2(0,T;H^{1+r})} & \leq C T^{\epsilon'} \mid \vite^n-\vite^{n-1} \mid_{K^{r+2}},                        \\
\mid \extras^{n+1} -\extras^n\mid_{H^{\frac{1+r}{2}}(0,T;L^2)} & \leq C T^{\epsilon'} \mid \vite^n-\vite^{n-1}\mid_{K^{r+2}}, \\
\mid \extras^{n+1} -\extras^n\mid_{H^{\frac{r}{2}}(0,T;H^1)} & \leq C T^{\epsilon'} \mid \vite^n-\vite^{n-1} \mid_{K^{r+2}}   ,                     \\
\mid \extras^{n+1} -\extras^n \mid_{H^{\frac{r}{2}+\frac{1}{4}}(0,T;H^{\frac{1}{2}})} & \leq C T^{\epsilon'}\mid \vite^n-\vite^{n-1} \mid_{K^{r+2}}.
\end{array}
\]
where $\epsilon'$ is a positive constant not depending on $T_0$ nor $n$.

In addition, for $(0,0,0,\extras^{n+1})$, we need to prove also 

\[
\mid \extras^{n+1} -\extras^n \mid_{H^{\frac{1+r}{2}}(0,T;H^{1+r})} \leq C T^{\epsilon'}\mid \vite^n-\vite^{n-1} \mid_{K^{r+2}},
\]
which contains all the others. So we will prove only the latter.\\

As a first step for all the preceding estimates, we need uniform in $n$ and in
$T\leq T_0$ estimates so as to manage nonlinearities:
\begin{eqnarray}
\label{2ch24}
\exists  \,  T, V, S \, / \mid \extras^n \mid_{H^{1+r}} (t) & \leq S & \forall n ,0<t<T \\
\label{2ch25}
\mid \vite^n, q^n, \phi^n, 0 \mid_{X_T^r} & \leq V & \forall n.
\end{eqnarray}

This will be done in Subsection \ref{sub33} after some lemmas stated in
Subsection \ref{sub32}.

As a second step, we prove in Subsection \ref{sub34} that the
sequence is of Cauchy type thanks to the fact that $C$ is independent
of $T\leq T_0$ and to the $T^{\epsilon'}$ term:

\begin{eqnarray}
\label{2ch34.1}
\bullet \mid \extras^{n+1} -\extras^n\mid_{H^{\frac{1+r}{2}}(0,T;H^{1+r})} &\leq & C T^{\epsilon'}\mid \extras^{n+1} -\extras^n\mid_{H^1(0,T;H^{1+r})},                       \\
\label{2ch34.2}
 & \leq & C T^{\epsilon'} \mid \gradu^n-\gradu^{n-1} \mid_{L^2(0,T;H^{1+r})}           \\
\label{2ch34.3}
 & \leq & C T^{\epsilon'} \mid \vite^n-\vite^{n-1} \mid_{K^{r+2}},
\end{eqnarray}

Among these inequalities, (\ref{2ch34.3}) is obvious. The inequality
(\ref{2ch34.1}) will be a simple
consequence of Lemma \ref{GAcinqtrois} proved hereafter that relies on
$(\extras^{n+1} -\extras^n)(t=0)=0$. The last inequality will be proved
in Subsection \ref{sub34}. After lifting some rhs, this will solve
(\ref{2ch7}).

$ $ \\

Coming back to the proof of Theorem \ref{thlineaire}, we will lift the
initial conditions with $\vite_1, q_1, \phi_1, \extras_1$ in Section
\ref{sub35}. We will also justify that $\extras_1 \in K^{1+r} \bigcap
L^{\infty}(0,T;H^{1+r})$ to apply Theorem
\ref{resolP2_simplifie}. Then the field $(\vite_1+\vite, q_1+q,
\phi_1+\phi,\extras_1+\extras)\in X_T^r$ will solve (\ref{eq15}) with
non-vanishing initial conditions.\\

In Subsection \ref{sub36}, we prove uniqueness of the solution of
(\ref{eq15}). The estimates enable to prove that the inverse of $P_1$
is bounded independently of $T<T_0$. This will
complete the proof of Theorem \ref{thlineaire}

\subsection{Some useful lemmas}
\label{sub32}

\begin{lemma}
\label{GAquatredeux}
Let $X$ be a Hilbert space, $0\leq s \leq 2$, such that $s-\Frac{1}{2}$
is not integer.\\
There exists a bounded extension operator from
$\left\{\vite \in H^s(0,T;X), \; \newline \partial_t^k \vite (0) =0,
\; 0\leq k < s-\Frac{1}{2} \right\}$ in $H^s(\mathbb{R}^+;X)$.
The constant $C$ of its boundedness does not depend on $T\leq T_0$.
\end{lemma}

The proof of this fundamental lemma based on reflections is
classical and can be found either in \cite{LionsMagenes} p. 13 or in
\cite{JTBeale81} (Lemma 2.2).

\begin{remark}
\label{remarque1}
The initial vanishing conditions (even up to some derivatives) are
needed whatever the extension operator (see remark
\ref{remarque2}). This is the reason why we have to solve the system
with vanishing initial conditions. For such functions the constants
do not depend on $T<T_0$. Then we lift the initial conditions
to solve the full system.
\end{remark}

\begin{lemma}
\label{GAcinqtrois}
Let $X$ be a Hilbert space, $1/2< r\leq 1,  \; 0\leq s \leq r$ and
$s\neq \Frac{1}{2}$. If $f \in H^r(0,T;X)$ and $f(0)=0$, then ;
\[ \mid f \mid_{H^s(0,T;X)} \leq C T^{r-s} \mid f \mid_{H^r(0,T;X)}.\]
The constant $C$ does not depend on $T\leq T_0$.
\end{lemma}

This lemma is proved in \cite{G.Allaintoulouse} (Lemma 5.3). The proof
is based on extension and interpolation. We need to have $f(0)=0$ to
have $C$ independent of $T\leq T_0$. Indeed,
$|e^{-t}|_{L^2(0,T)}/|e^{-t}|_{H^1(0,T)}$ can be bounded by a constant
that does even not depend on $T<T_0$.

We will repeatedly need the following lemma about identity operator
from $K^r$ to the space $H^p(0,T;H^{r-2p}(\Omega))$:

\begin{lemma}
\label{GAcinqquatre}
Suppose $0\leq r \leq 4$ and $p\leq \Frac{r}{2}$.
\begin{itemize}
\item[(i)] The identity extends to a bounded operator $K^r(\Omega\times(0,T)) \rightarrow H^p(0,T;H^{r-2p}(\Omega))$.
\item[(ii)] If $r$ is not an odd integer, the restriction of this operator to the subspace:
\[
\left\{v \in K^r(\Omega\times(0,T)) \; /\; \partial_t^k v(0)=0, \; \forall k \; 0\leq k <\Frac{r-1}{2}\right\}
\]
is bounded independently of $T<T_0$.
\end{itemize}
\end{lemma}

\begin{proof}

$ $ \\[-8mm]

\begin{itemize}
\item[(i)] From the definition of $K^r$, identity can be defined from
$K^r$ to $L^2(0,T;H^{r})$ and to $H^{\frac{r}{2}}(0,T;L^2)$. It is
bounded. So identity can be defined on the interpolate space between
$K^r$ and $K^r$ (which is $K^r$ !) to:
\[
\left[ L^2(0,T;H^r),H^{\frac{r}{2}}(0,T;L^2)\right]_{\theta}=H^{\frac{\theta r}{2}}(0,T;H^{r(1-\theta)}),
\]
from Proposition 4,2.3 of \cite{LionsMagenes}. It is the announced result with $p=\theta r/2$.
\item[(ii)] Lemma \ref{GAquatredeux} guarantees that the extension
to $t>T$ is bounded
independently of $T<T_0$. Then one may apply the same proof as above
on $[0,+\infty] \times \Omega$. The norms are bounded independent of $T<T_0$.
\end{itemize}
\end{proof}

\subsection{Uniform estimates}
\label{sub33}

We want to prove (\ref{2ch24}, \ref{2ch25}) by induction. Indeed we
will prove a more general result. The inductive hypothesis writes for
$(\vite^n,q^n,\phi^n,\extras^n)$ defined by (\ref{2ch16}, \ref{2ch17}):
\begin{equation}
\label{2ch35}
\exists T_0, V<1, S \mbox{ positive such that }  \forall n
\left\{
\begin{array}{l}
\mid \extras^n \mid_{H^{1+r}}(t) \leq S,\\
\mid \vite^n, q^n, \phi^n, 0 \mid_{X_T^r} \leq V, \\
\mid \gradu^n \mid_{L^2(0,T_0;H^{1+r})} \leq V, \\
C \mid \gradu_1 \mid_{L^2(0,T_0;H^{1+r})} \leq 1, \\
C\sqrt{T_0}(V+\mid \gradu_1 \mid_{L^2(0,T_0;H^{1+r})}) <1,\\
C (V+\mid \Mga \mid_{L^2(0,T_0;H^{1+r})}) <S,\\
C\mid \fsec \mid_{H^{r,\frac{r}{2}}(0,T_0)} < V/5\\
CT^{\epsilon'}(V+\mid \Mga \mid_{L^2(0,T_0;H^{1+r})}) < V/5.\\
\end{array}
\right.
\end{equation}
Notice that the five last inequalities do not depend on $n$. So it
suffices to determine a $T_0$ at $n=0$ to satisfy these
properties. Notice also that the $L^{\infty}(0,T;H^{1+r})$ estimate on
$\extras^{n+1}$ cannot be omited in the hypothesis because of the
nonlinear terms in the constitutive equation which will require even
more regularity. Notice also that (\ref{2ch35})$_3$ is included in
(\ref{2ch35})$_2$, but more useful.\\

Obviously the property at $n=0$ is true if $T_0$ is small enough.

Let us assume property (\ref{2ch35}) be true for $n$ and let us try to
deduce that it is true for $n+1$. We must be careful that the
constants do depend neither on $T<T_0$ nor on $n$.\\

In a first step, we want to prove (\ref{2ch24}) or (\ref{2ch35})$_1$
from the $H^{1+r}$ scalar product of (\ref{2ch16}) with
$\extras^{n+1}$. Thanks to the fact that $g_a$ is only quadratic, it
provides (with a $C$ independent of $T<T_0$ and of $n$):

\begin{equation}
\label{2ch34bis}
\begin{array}{c}
\mid \extras^{n+1}\mid_{1+r}^2 +\Frac{\We}{2}\DT{\mid \extras^{n+1}\mid_{1+r}^2} \leq C \left[ \rule[5mm]{0cm}{0cm} 2\varepsilon\mid \gradu^n \mid_{1+r}+\mid \Mga \mid_{1+r} + \right.                                \\
+\left. C\We(\mid \gradu^n \mid_{1+r}\mid \extras^{n+1}\mid_{1+r} +\mid \gradu_1 \mid_{1+r} \mid \extras^{n+1}\mid_{1+r} + \mid \gradu^n \mid_{1+r}\mid \extras_1 \mid_{1+r})\rule[5mm]{0cm}{0cm} \right] \mid \extras^{n+1} \mid_{1+r},
\end{array}
\end{equation}
and so
\[
\begin{array}{c}
(1-C\We(\mid \gradu^n \mid_{1+r}+\mid \gradu_1 \mid_{1+r}))\mid \extras^{n+1}\mid_{1+r}^2+\Frac{\We}{2}  \DT{\mid \extras^{n+1}\mid_{1+r}^2} \leq                    \\
C\left[(2\varepsilon+C \We \mid \extras_1 \mid_{1+r})\mid \gradu^n \mid_{1+r}+\mid \Mga \mid_{1+r} \right] \mid \extras^{n+1} \mid_{1+r}.
\end{array}
\]
If we choose not to simplify $\mid \extras^{n+1} \mid_{1+r}$,
H\"older inequality enables to state:
\[
\begin{array}{c}
\DT{\left(e^{2\Int_0^t\Frac{(1/2-C\We (\mid \gradu^n \mid_{1+r}+\mid \gradu_1 \mid_{1+r}))}{\We}ds}\mid \extras^{n+1}\mid_{1+r}^2(t)\right)}\leq                   \\
C(\varepsilon,\extras_1,\We)\left(  \mid \gradu^n \mid_{1+r}^2+\mid \Mga \mid_{1+r}^2 \right)e^{2\Int_0^t\Frac{(1/2-C\We (\mid \gradu^n \mid_{1+r}+\mid \gradu_1 \mid_{1+r}))}{\We}{\rm d}s}.         
\end{array}
\]

Thanks to the assumption that $\extras_1 \in L^{\infty}(0,T;H^{1+r})$
 and that $\extras^{n+1}(t=0)=0$, one may deduce:
\[
\begin{array}{l}
\mid \extras^{n+1}\mid_{1+r}^2(t) \leq                                \\
\Int_0^t (e^{-2\Int_s^t\Frac{(1/2-C\We (\mid \gradu^n \mid_{1+r}+\mid \gradu_1 \mid_{1+r}))}{\We}dt'}(C \mid \gradu^n \mid_{1+r}^2+\mid \Mga \mid_{1+r}^2))ds.
\end{array}
\]

We have to estimate the term in the exponential. Since $s\leq t$ and
thanks to (\ref{2ch35}) for $n$:

\[
\begin{array}{c}
-2\Int_s^t \Frac{1/2-C\mbox{We}(\mid \gradu^n \mid_{1+r}+\mid \gradu_1 \mid_{1+r})}{\mbox{We}}ds' \leq C\Int_s^t \mid \gradu^n\mid_{1+r}+\mid \gradu_1 \mid_{1+r} ds'     \\
\mbox{\hspace{4cm}} \leq C \sqrt{T} \left( \mid \gradu^n \mid_{L^2(0,T;H^{1+r})}+\mid \gradu_1 \mid_{L^2(0,T;H^{1+r})}\right) <1.
\end{array}
\]

This enables to state:
\begin{equation}
\label{2ch35.5}
\mid \extras^{n+1} \mid_{1+r}(t) \leq C \left(\mid \nabla \vite^n\mid_{L^2(0,T;H^{1+r})}+\mid \Mga \mid_{L^2(0,T;H^{1+r})}\right)\leq C\left( V+\mid \Mga \mid_{L^2(0,T_0;H^{1+r})}\right),
\end{equation}
and we deduce the announced (\ref{2ch24}) or (\ref{2ch35})$_1$ for suitable $S$ not depending on $n$. Moreover we have:
\begin{equation}
\label{2ch36}
\mid \extras^{n+1}\mid_{L^2(0,T;H^{1+r})}
\leq C T^{\epsilon'}(V+\mid \Mga \mid_{L^2(0,T;H^{1+r})})<S.
\end{equation}

\begin{remark}
The proof above relies on the fact that the $C$ found at $n+1^{th}$ iteration
 does not depend on the $C$ assumed at the $n^{th}$ iteration.
\end{remark}

In a second step, so as to have uniform in $n$ estimates on $\mid \vite^n
\mid_{K^{2+r}}$ (denoted (\ref{2ch35})$_2$ and (\ref{2ch35})$_3$),
we need more regular in time estimates on
$\extras$. More precisely we need $\mid
\extras^{n+1}\mid_{H^1(0,T;H^{1+r})}$. 

We take back
(\ref{2ch16}), isolate the time derivative and take the $H^{1+r}(\Omega)$ norm:
\[
\begin{array}{c}
\left| \DP{\extras^{n+1}}{t} \right|_{1+r}(t) \leq \Frac{2\varepsilon}{\We}\mid \gradu^n \mid_{1+r}+\Frac{1}{\We}\mid \Mga \mid_{1+r}+\Frac{1}{\We}\mid \extras^{n+1} \mid_{1+r}+            \\
\mbox{\hspace{1cm}}+C\left(\mid \gradu^n\mid_{1+r} \mid \extras^{n+1} \mid_{1+r}+\mid \gradu_{1}\mid_{1+r} \mid \extras^{n+1} \mid_{1+r}+\mid \gradu^n\mid_{1+r} \mid \extra_{1} \mid_{1+r} \right).
\end{array}
\]
One may then take the $L^2_t$ norm and it writes thanks to
(\ref{2ch35}), (\ref{2ch35.5}) and (\ref{2ch36}):
\[
\begin{array}{rl}
\left|\DP{\extras^{n+1}}{t}\right|_{L^2(0,T;H^{1+r})} \leq & C(\varepsilon, \We) \left(\mid \gradu^n \mid_{L^2(0,T;H^{1+r})} +\mid \Mga \mid_{L^2(0,T;H^{1+r})}\right)+\\
& +C|\extras^{n+1}|_{L^2(0,T;H^{1+r})}+C|\gradu^n|_{L^2(0,T;H^{1+r})}|\extras^{n+1}|_{L^{\infty}(0,T;H^{1+r})}+\\
& +C|\gradu_1|_{L^2(0,T;H^{1+r})}|\extras^{n+1}|_{L^{\infty}(0,T;H^{1+r})}+\\
& +C|\gradu^n|_{L^2(0,T;H^{1+r})}|\extras_1|_{L^{\infty}(0,T;H^{1+r})}.
\end{array}
\]
It is then easy to use (\ref{2ch36}) to prove
\begin{equation}
\label{2ch36bis}
\mid \extras^{n+1}\mid_{H^1(0,T;H^{1+r})}\leq C
(\varepsilon, \We, \vite_1, \extras_1,S)( \mid \gradu^n \mid_{L^2(0,T,H^{1+r})}+
\mid \Mga \mid_{L^2(0,T,H^{1+r})}).
\end{equation}

This estimate is even stronger than the estimates required in
$L^2(0,T;H^{1+r})$, $H^{\frac{r}{2}}(0,T;H^{1})$, and
$H^{\frac{r}{2}+\frac{1}{4}}(0,T;H^{\frac{1}{2}})$. Because
$\extras^{n+1}(t=0)=0$ we may use Lemma
\ref{GAcinqtrois} and Lemma \ref{GAcinqquatre} which imply that the constant $C$ below is
independent of $T\leq T_0$:
\begin{equation}
\label{2ch36ter}
\mid \extras^{n+1}\mid_{H^{\frac{r}{2}}(0,T;H^1)} \leq C T^{\epsilon'}\mid \extras^{n+1}\mid_{H^1(0,T;H^1)} \leq C T^{\epsilon'}(\mid \gradu^n \mid_{L^2(0,T,H^{1+r})}+\mid \Mga \mid_{L^2(0,T,H^{1+r})}).
\end{equation}
Similarly, Lemma \ref{GAcinqtrois} and Lemma \ref{GAcinqquatre} enable
to get:
\begin{equation}
\label{2ch36quater} \mid \extras^{n+1}\mid_{H^{\frac{r}{2}+\frac{1}{4}}(0,T;H^{\frac{1}{2})}}  \leq C T^{\epsilon'}\mid \extras^{n+1}\mid_{H^1(0,T;H^{\frac{1}{2}})} \leq C T^{\epsilon'}(\mid \gradu^n \mid_{L^2(0,T;H^{1+r})}+ \mid \Mga \mid_{L^2(0,T;H^{1+r})})
\end{equation}

We are in position to get estimates on $\mid \vite^{n+1}
\mid_{K^{2+r}}$ independently of $n$ and in particular
(\ref{2ch35})$_2$. Indeed we use the Theorem 4.1 of G. Allain
\cite{G.AllainAMO1987} with $g=-\extras^{n+1} \cdot \Norm$,
$k=\vite_0=a=0$, $\fsec=\fsec+\mbox{div } \extras^{n+1}$ to find
solutions to (\ref{2ch17}). Of course we use also the already proved
estimates (\ref{2ch36}, \ref{2ch36ter}, \ref{2ch36quater}) together
with classical trace theorems and the induction hypothesis
(\ref{2ch35}). If $0<r<1/2$, it enables:

\begin{equation}
\label{2ch38}
\begin{array}{l}
\mid \vite^{n+1}, q^{n+1}, \phi^{n+1},0 \mid_{X_T^r}\leq C \mid \fsec + \mbox{div} \extras^{n+1} \mid_{K^r} +C\mid \extras^{n+1} \mid_{K^{r+\frac{1}{2}}(S_F \times (0,T))}                     \\
\leq C\mid \fsec \mid_{K^r}+C \mid \extras^{n+1} \mid_{L^2(0,T;H^{1+r})\bigcap H^{\frac{r}{2}}(0,T;H^1)} +C  \mid \extras^{n+1} \mid_{L^2(0,T;H^{1+r}(\Omega))\bigcap H^{\frac{r}{2}+\frac{1}{4}}(0,T;H^{\frac{1}{2}}(\Omega))}                                                    \\
\leq C\mid \fsec \mid_{K^r}+CT^{\epsilon'}(V+\mid \Mga \mid_{L^2(0,T;H^{1+r})})+CT^{\epsilon'}(V+\mid \Mga \mid_{L^2(0,T;H^{1+r})})+      \\
\mbox{\hspace{3cm}}+CT^{\epsilon'}(V+\mid \Mga \mid_{L^2(0,T;H^{1+r})})+CT^{\epsilon'}(V+\mid \Mga \mid_{L^2(0,T;H^{1+r}(\Omega))})\\
\leq V.
\end{array}
\end{equation}

In the inductive hypothesis we assumed $V<1$ and $S$ whatsoever.
Should we take a smaller $T_0$ (the constants $C$ above depend on
$\varepsilon, V, S, \We, \mbox{Re}, \extras_1, T_0$ but not on $T\leq
T_0$ nor on $n$) we may derive (\ref{2ch35})$_2$ and (\ref{2ch35})$_3$
for $n+1$ assuming the inductive hypothesis for $n$.

Since the inequalities (\ref{2ch35})$_4$ to (\ref{2ch35})$_{8}$ even do
not depend on $n$, they are satisfied from the fact that they are
satisfied at $n=0$. So the inductive proof of (\ref{2ch35}) is
complete and (\ref{2ch24}, \ref{2ch25}) are also proved.

The case of more general constitutive laws is studied in Appendix A.

\subsection{Convergence of the sequence}
\label{sub34}

We intend to prove that the sequence $(\vite^n, q^n, \phi^n,
\extras^n)$ is a Cauchy sequence in the Banach space $X_T^r$. Although
we proved (\ref{2ch24}-\ref{2ch25}) and even (\ref{2ch35}), 
we will only use:
\[
\begin{array}{c}
\mid \vite^n \mid_{K^{2+r}(\Omega\times(0,T))} \leq V \; \forall n, \\
\mid \extras^n \mid_{L^{\infty}(0,T;H^{1+r})} \leq S \; \forall n.
\end{array}
\]
We still need the estimate (\ref{2ch34.2}) that will be proved
hereafter. Notice that our domain is unboudned. This would make
compactness argument very difficult to use.

In order to prove convergence, we take the difference of
(\ref{2ch16}) at the $n^{th}$ and $n+1^{th}$ iteration :

\begin{equation}
\label{2ch41}
\begin{array}{c}
\extras^{n+1}-\extras^n +\We \left( \DP{(\extras^{n+1}-\extras^n)}{t}-g(\gradu^n,\extras^{n+1}-\extras^n)-g(\gradu^n-\gradu^{n-1},\extras^n)- {\vrule width 0mm height 0.6cm}\right. \\
\left. {\vrule width 0mm height 0.6cm} g(\gradu_1,\extras^{n+1}-\extras^n)-g(\gradu^n-\gradu^{n-1},\extras_1) \right)= 2 \varepsilon \symet[\vite^{n}-\vite^{n-1}].
\end{array}
\end{equation}

We want to prove (\ref{2ch34.2}). The uniform estimates
(\ref{2ch24}-\ref{2ch25}) and assumptions on $\extras_1, \gradu_1$
enable to simplify the scalar product in $H^{1+r}$ of the equation
(\ref{2ch41}) with $\extras^{n+1}-\extras^n$:
\[
\begin{array}{l}
\mid \extras^{n+1}-\extras^n \mid_{1+r}^2+\Frac{\We}{2}\DT{\mid \extras^{n+1}-\extras^n \mid_{1+r}^2} \leq 2\varepsilon \mid \gradu^n-\gradu^{n-1}\mid_{1+r}\mid \extras^{n+1}-\extras^n \mid_{1+r}+          \\
+C\We \left[ \rule[5mm]{0cm}{0cm} \mid \gradu^n \mid_{1+r}\, \mid \extras^{n+1}-\extras^n \mid_{1+r}^2 + \mid \gradu^n-\gradu^{n-1}\mid_{1+r} \mid \extras^n \mid_{1+r} \mid \extras^{n+1}-\extras^n \mid_{1+r} + \right.\\
\left.+ \mid \gradu_1 \mid_{1+r} \mid \extras^{n+1}-\extras^n \mid_{1+r}^2 +\mid \gradu^n-\gradu^{n-1}\mid_{1+r} \, \mid \extras_1 \mid_{1+r} \,\mid \extras^{n+1}-\extras^n\mid_{1+r}  \right]\\[3mm]
\Rightarrow (1/2-C\We(\mid \gradu^n\mid_{1+r}+\mid \gradu_1 \mid_{1+r}))\mid \extras^{n+1}-\extras^n \mid_{1+r}^2+\Frac{\We}{2} \DT{\mid \extras^{n+1}-\extras^n \mid_{1+r}^2} \leq                    \\
\mbox{\hspace{6cm}} C(\varepsilon,\We, V,S, \extras_1) \mid \gradu^n-\gradu^{n-1}\mid_{1+r}^2.
\end{array}
\]

As in the proof of (\ref{2ch36}) from (\ref{2ch34bis}) one may
use (\ref{2ch35}) to have:
\begin{equation}
\label{2ch42}
\mid \extras^{n+1}-\extras^n \mid_{1+r}(t) \leq C(\varepsilon, \We, V, S, \extras_1) \mid \gradu^n-\gradu^{n-1}\mid_{L^2(0,T; H^{1+r})}.
\end{equation}

We want to prove (\ref{2ch34.2}). Since we already know that $\mid
\extras^{n} \mid_{L^{\infty}(0,T;H^{1+r})} \leq S$, and $\mid
\extras_1 \mid_{L^{\infty}(0,T_0;H^{1+r})}$ and $\mid \gradu_1
\mid_{L^2(0,T_0;H^{1+r})\bigcap H^{\frac{1+r}{2}}(0,T_0;L^2)}$ are in
a bounded ball (here we need this to have continuity for bounded
initial conditions because $\vite_1, \extras_1$ will depend on them),
we may use the inequality (\ref{2ch42}) and come back to (\ref{2ch41})
whose $L^2(0,T;H^{1+r})$ norm provides:

\[
\begin{array}{rl}
\left| \DP{(\extras^{n+1}-\extras^n)}{t} \right|_{L^2(0,T;H^{1+r})} \leq & \mid \extras^{n+1}-\extras^n \mid_{L^2(0,T;H^{1+r})} + 2\varepsilon \mid  \gradu^n-\gradu^{n-1}   \mid_{L^2(0,T;H^{1+r})} +\\
& +C \mid \gradu^{n}\mid_{L^2(0,T;H^{1+r})} \mid \extras^{n+1}-\extras^{n} \mid_{L^{\infty}(0,T;H^{1+r})}    +      \\
& +C(S) \mid \gradu^n-\gradu^{n-1}\mid_{L^2(0,T;H^{1+r})} +\\
& +C \mid \gradu_1\mid_{L^2(0,T;H^{1+r})} \mid \extras^{n+1}-\extras^n \mid_{L^{\infty}(0,T;H^{1+r})} + \\
& +C \mid \extras_1\mid_{L^{\infty}(0,T;H^{1+r})} \mid \gradu^n-\gradu^{n-1}\mid_{L^2(0,T;H^{1+r})},\\
\leq & C(\varepsilon, \We, V, S,\vite_1, \extras_1) \mid \gradu^n-\gradu^{n-1}\mid_{L^2(0,T;H^{1+r})}.
\end{array}
\]

Notice that the constant $C(\varepsilon, \We, V, S,\vite_1,
\extras_1)$ depends on $\vite_1, \extras_1$ only through a bound of
their norm. The latter can be combined with (\ref{2ch42}) to get the
requested inequality:
\begin{equation}
\label{2ch43}
\left| \extras^{n+1}-\extras^n \right|_{H^1(0,T;H^{1+r})} \leq C \mid \gradu^n-\gradu^{n-1}\mid_{L^2(0,T;H^{1+r})}.
\end{equation}

\subsection{Passing to the limit}

Since estimates (\ref{2ch34.1}) to (\ref{2ch34.3}) are proved for
$(0,0,0,\extras^n)$ and estimate (\ref{2ch17.5}) for
$(\vite^n,q^n,\Phi^n,0)$, the whole sequence $(\vite^n, q^n, \phi^n,
\extras^n)$ is of Cauchy type, should it be for sufficiently small
$T_0$. So it converges in $X_T^r$ to some $(\vite, q, \phi,
\extras)$.\\

Can we take the limit $n \rightarrow +\infty$ ? When the terms are
linear it is straightforward. When they are of the shape $\nabla
\vite^n \; \extras^n$ in the constitutive law, we remember that
$\nablah \vite^n $ converges in $K^{1+r}$ and $\extras^n$ in
$H^{\frac{1+r}{2}}(0,T;H^{1+r})$. So
\[
\begin{array}{rl}
\mid \gradu^n \extras^n -\gradu \extras \mid_{K^{1+r}} \leq & C \mid \gradu^n-\gradu \mid_{K^{1+r}} \mid \extras^n \mid_{H^{\frac{1+r}{2}}(0,T;H^{1+r})} +\\
& +C \mid \gradu \mid_{K^{1+r}} \mid \extras^n -\extras \mid_{H^{\frac{1+r}{2}}(0,T;H^{1+r})}.
\end{array}
\]
Notice that {\em a priori} we use the bound $C$ of the embedding
$H^{\frac{1+r}{2}}(0,T) \hookrightarrow L^{\infty}(0,T)$ that depends
on $T < T_0$. But we prove further Lemma \ref{lemma57} which enables to
state that the bound can be taken independent of $T <T_0$ because
$(\gradu^n -\gradu)(t=0)=0$ and $\extras^n(t=0)=0$. So $\gradu^n
\extras^n$ tends to $\gradu \, \extras$ in $K^{1+r}$. Even if more
complex nonlinearities of $\gradu$ and $\extras$ had been chosen,
they could pass to the limit provided uniform
$H^{\frac{1+r}{2}}(0,T;H^{1+r})$ estimates of $\extras^n$ can be
derived. So the most difficult step for more complex constitutive
laws is the uniform in $n$ estimate. It is derived for some laws
in Appendix A.

As a conclusion, the volumic and boundary equations of
$P_2[\vite_1,\extras_1](\vite,q,\phi,\extras)=(\fsec,0,\Mga,0,0,0,0)$
(see (\ref{eq16}) for the definition of $P_2$) are satisfied by the
limit of the sequence.\\

Concerning the initial velocity, $\vite^n$ converges to $\vite$ in
$H^{\frac{1}{2}+\frac{\theta}{2}}(0,T;H^{1+r-\theta})$ with
$0<\theta<r$ (cf. Lemma \ref{GAcinqquatre}). So we have proved that
$\vite(t=0)=0$ in $H^{1+r-\theta}$ and so also in $H^{1+r}$.

Concerning the initial extra stress, $\extras^n$ converges to $\extras$
in $H^{\frac{1+r}{2}}(0,T;H^{1+r})$. So $\extras(0)=0$ in $H^{1+r}$.

Up to now we have proved that (\ref{2ch7}) has a
solution. By lifting the initial conditions, we will solve
(\ref{eq15}). This is done in the next subsection.

\subsection{Lift the initial conditions to solve $P_1$}
\label{sub35}

To solve the linearized problem about zero deformation
(${\boldsymbol \xi}=0$), with general rhs, we introduced and solved
 the reduced $P_2$ problem (\ref{2ch7}). We want now to solve $P_1$ (or the equation
(\ref{eq15})). So let $(\fsec, a, \Mga, g, k , \vite_0 , \extras_0 )$
in a closed ball of radius $R$ in $Y_T^r$. We also take $\fsec_1 \in
K^r$ and $\Mga_1 \in K^{1+r}$.

We start by lifting the initial extra stress. So let $\extras_1$ be such that:
\begin{equation}
\label{2ch47}
\left\{
\begin{array}{c}
\extras_1 +\We \Frac{ \partial \extras_1}{\partial t} =\Mga_1           \\
\extras_1(0,X)=\extras_0(X) \; \; \forall X.
\end{array}
\right.
\end{equation}

One may easily check that $\extras_1 \in L^{\infty}(0,T_0;H^{1+r})$ 
for any $T_0$. This $\extras_1$ and its special regularity
has been repeatedly used in the estimates because of the
nonlinearities.

Then applying Theorem 4.1 of \cite{G.AllainAMO1987} to the rhs
$(\fsec_1,a,0,g-\extras_1. \Norm,k,\vite_0,0)$ provides $(\vite_1, q_1, \phi_1,0)$ in $X_T^r$ such that

\begin{equation}
\label{2ch48}
\left\{
\begin{array}{rll}
\mbox{Re}\vite_{1,t}-(1-\varepsilon)\Delta u_1+\nablah q_1 & =\fsec_1 &\mbox{ in } \Omega,\\
\mbox{div} \vite_1 & =a & \mbox{ in } \Omega,\\
-q_1 \Norm +2(1-\varepsilon)\symet[\vite_1] \cdot \Norm -\alpha \partial_{\tang}(\phi_1 \Norm) & =g-\extras_1 \cdot \Norm & \mbox{ on }S_F\times(0,T)            \\
\phi_{1,t}-\partial_{\tang} \vite_1 \cdot \Norm & =k & \mbox{ on }S_F\times(0,T)   \\
\phi_1(t=0)& =0 & \mbox{ on }S_F   \\
\vite_1 & =0 &\mbox{ on }S_B\times(0,T),       \\
\vite_1(t=0)& =\vite_0(X) & \mbox{ in } \Omega.
\end{array}
\right.
\end{equation}

The just found $\vite_1, \extras_1$ are those denoted the same way as in
(\ref{2ch7}). To summarize, these fields satisfy
\[
P_1(\vite_1,q_1,\phi_1,\extras_1)=(\fsec_1-\mbox{ div }\extras_1,a,\Mga_1-2\varepsilon \symet[\vite_1]-\We  g_a(\gradu_1,\extras_1),g,k,\vite_0,\extras_0).
\]
and so $(\vite_1+\vite,q_1+q,\phi_1+\phi,\extras_1+\extras) \in X_T^r$
is such that 
\[
\begin{array}{l}
P_2[\vite_1,\extras_1](\vite,q,\phi,\extras)+P_1(\vite_1,q_1,\phi_1,\extras_1)=\\
P_1(\vite_1+\vite,q_1+q,\phi_1+\phi,\extras_1+\extras)=\\
\hspace{2cm}(\fsec+\fsec_1-\mbox{ div }\extras_1,a,\Mga+\Mga_1-2\varepsilon \symet[\vite_1]-\We  g_a(\gradu_1,\extras_1),g,k,\vite_0,\extras_0).
\end{array}
\]

So, by choosing $\fsec_1=\mbox{ div }\extras_1$ and $\Mga=2\varepsilon
\symet[\vite_1]+\We  g_a(\gradu_1,\extras_1)$,
we provide a solution
$(\vite_1+\vite,q_1+q,\phi_1+\phi,\extras_1+\extras)$ of (\ref{eq15})
as announced in Theorem \ref{thlineaire}. Notice that we cannot set
$\Mga_1:=2\varepsilon \symet[\vite_1]+\We 
g_a(\gradu_1,\extras_1)$ since $\vite_1, \extras_1$ depend on $\Mga_1$.
It would be only
an implicit equation. Then it suffices to let $\Mga_1$ be whatsoever
and determine $\Mga$ as a function of $\vite_1,\extras_1$ as done here.

To complete the proof of Theorem \ref{thlineaire}, we still must
prove the inverse of $P_1$ is bounded.

\subsection{Uniqueness and boundedness}
\label{sub36}

To prove uniqueness of the solutions to (\ref{eq15}) and the
boundedness of $P_1^{-1}$, we start from the equation (\ref{eq15})
satisfied by $(\vite, q, \phi, \extras)\in X_{T_0}^r$ for given
$(\fsec, a, \Mga, g, k, \mbox{} \vite_0, \mbox{} \extras_0 )$ in a
closed ball $B$ in $Y_T^r$. First we prove uniform estimates and then
boundedness of $P_1^{-1}$.

\subsubsection{Uniform estimates}

\label{uniform_est_P1}

To have uniform estimates, we may pass to the limit in
(\ref{2ch35}) and recover the estimate for the velocity
\begin{equation}
\label{2ch48.25}
| (\vite,q,\Phi,0)|_{X_{T_0}^r}\leq C,\; \; 
| \gradu |_{L^2(0,T;H^{1+r})} \leq C.
\end{equation}

and the extra stress
\begin{equation}
\label{2ch48.5}
| \extras |_{1+r}(t) \leq S.
\end{equation}
These bounds are true only for initially vanishing fields. But then it
suffices to add a lift function to have the same bounds for general
$(\vite, q, \phi, \extras)$ solution of the system (\ref{eq15}).\\

It is not difficult to mimic the proof of the inequality
(\ref{2ch34.2}) made in Subsection \ref{sub34} to prove even a
stronger result. The bound (\ref{2ch48.5}) enables to prove that the partial
derivative in time of $\extras$ is also bounded in
$L^2(0,T;H^{1+r})$. Indeed, the nonlinearities can be managed thanks
to (\ref{2ch35.5}). Similarly to
(\ref{2ch36bis}), it is then easy to get that:
\[
\mid \extras \mid_{H^1(0,T;H^{1+r})} \leq C \left( \mid \Mga \mid_{L^2(0,T;H^{1+r})} + \mid \gradu \mid_{L^2(0,T;H^{1+r})} +S \right).
\]

\subsubsection{Boundedness}

Let us assume (\ref{eq15}) is satisfied by $(\vite, p, \phi, \extras)$
and $(\vite', p', \phi', \extras')$ with two right-hand sides.

\begin{equation}
\label{2ch49}
\begin{array}{ll}
\mbox{Re}\, \DP{(\vite-\vite')}{t}-(1-\varepsilon)\Delta (\vite-\vite') +\nablah (q-q') = \fsec-\fsec' +\mbox{div}( \extras-\extras') & \mbox{ in } \Omega\times(0,T),\\
\mbox{div} (\vite-\vite') =a-a' &  \mbox{ in } \Omega\times(0,T),\\
\extras-\extras' +\We \left(\DP{(\extras-\extras')}{t}- g_a(\gradu -\gradu',\extras)- g_a(\gradu',\extras-\extras')\right) =       &                       \\
\mbox{\hspace{3cm}}=2\varepsilon\symet[\vite-\vite'] + \Mga-\Mga' & \mbox{ in } \Omega\times(0,T) \\
(\extras-\extras') \cdot \Norm -( q-q') \Norm +2(1-\varepsilon)\symet[\vite-\vite']\cdot \Norm -\alpha\partial_{\tang}((\phi - \phi') \Norm)=&\\
\mbox{\hspace{3cm}} =g-g' & \mbox{ in } S_F\times (0,T),\\
(\phi-\phi')_t-\partial_{\tang}(\vite-\vite')\cdot \Norm =k-k' & \mbox{ in } S_F\times (0,T),\\
(\phi-\phi')(t=0)=0 &  \mbox{ on } S_F,\\
(\vite-\vite')(t=0)=(\vite_0-\vite'_0)(X) & \mbox{ in } \Omega,\\
(\extras-\extras')(t=0)=(\extras_0-\extras'_0)(X) & \mbox{ in } \Omega,\\
\vite -\vite'=0 & \mbox{ on } S_B \; \;\forall t.
\end{array}
\end{equation}

Theorem 4.1 of \cite{G.AllainAMO1987} gives:

\begin{equation}
\label{2ch67}
\begin{array}{l}
\mid \vite-\vite',q-q',\phi-\phi',0\mid_{X_T^r} \leq \\
\mbox{\hspace{0.5cm}}\leq C \mid \fsec-\fsec'+\mbox{div}(\extras-\extras'),a-a',0,g-g'-(\extras-\extras')\cdot \Norm,k-k',\vite_0-\vite'_0 , 0\mid_{ Y_T^r}\\
\mbox{\hspace{0.5cm}}\leq C \mid \fsec-\fsec',a-a',0,g-g',k-k',\vite_0-\vite'_0 , 0 \mid_{Y_T^r}+                         \\
\mbox{\hspace{3cm}}+C\mid \extras-\extras' \mid_{L^2(0,T;H^{1+r})\bigcap H^{\frac{r}{2}}(0,T;H^1)}+C\mid\extras-\extras' \mid_{K^{r+\frac{1}{2}}(S_F\times(0,T))}.
\end{array}
\end{equation}

Therefore we must estimate $\extras-\extras'$ in three norms
$L^2(0,T;H^{1+r}), H^{r/2}(0,T;H^1)$ and
$H^{\frac{r}{2}+\frac{1}{4}}(0,T;H^{1/2})$. We also need estimates of
$\extras -\extras'$ in $H^{\frac{1+r}{2}}(0,T;H^{1+r})$ to complete
the norm of $X_T^r$. The last norm is the strongest but {\em a priori}
no bound can be found for initially non-vanishing
functions. Nevertheless we will establish such bounds because we can
lift the initial conditions and then use classical results.

\begin{remark}
By various estimates, we can bound $| \extras -\extras' |$ with terms
including $| \extras_0 -\extras_0'|$ but also $| \gradu
-\gradu'|$. The latter term (and not the former) needs to have a
coefficient sufficiently small for $T_0$ small enough. This will
enable these terms to be absorbed by the left-hand sides of
(\ref{2ch67}) and so lead to the expected bound (\ref{2ch6}).
\end{remark}

From now on, we assume $(\vite,q,\phi,\extras)$ solution of
(\ref{eq15}) and so initially non-vanishing. If the functions
vanish initially, the bounds in
$L^2(0,T;H^{1+r})$, $H^{r/2}(0,T;H^1)$ and
$H^{\frac{r}{2}+\frac{1}{4}}(0,T;H^{1/2})$ are included in the
bound in $H^{\frac{1+r}{2}}(0,T;H^{1+r})$. So the hierarchy of these norms is
{\em a priori} not obvious. Yet we prove below that the
$H^{\frac{1+r}{2}}(0,T;H^{1+r})$ bound ensures the others. To apply Lemma
\ref{GAcinqquatre}, we need to lift the initial values. Then we
apply the process that led us to (\ref{2ch43}) to the new initially
vanishing functions. Moreover, we will even derive an estimate in
$H^1(0,T;H^{1+r})$.

Let $\extra$ be such that

\[
\begin{array}{rl}
\extra+\We\, \DP{\extra}{t} & =0  ,         \\
\extra(0)                 & =\extras_0,
\end{array}
\]
and in the same manner $\extra'(t,X)=e^{\frac{-t}{ \mbox{ \tiny
We}}}\extras'_0(X)$. Instead of (\ref{2ch49})$_3$, by renaming
$\extras:= \hat{\extras}+\extra$ and $\extras':= \hat{\extras}'+\extra'$,
we have:
\begin{equation}
\label{2ch69}
\begin{array}{l}
\hat{\extras}-\hat{\extras}' +\We \left(\DP{(\hat{\extras}-\hat{\extras}')}{t}- g_a(\gradu -\gradu',\hat{\extras})- g_a(\gradu',\hat{\extras}-\hat{\extras}')\right) =                              \\
\mbox{\hspace{1cm}} =2\varepsilon\symet[\vite-\vite'] +\Mga-\Mga' +\We \, g_a(\gradu -\gradu',\extra)+\We \, g_a(\gradu',\extra-\extra')  \; \mbox{ in } \Omega\times(0,T) ,
\end{array}
\end{equation}

and mainly $(\hat{\extras}-\hat{\extras}')(t=0)=0$. Moreover, all the
uniform bounds stated in Subsubsection \ref{uniform_est_P1} apply to
$\hat{\extras}$ and $\hat{\extras}'$. Now, we take the $H^{1+r}$ scalar
product of (\ref{2ch69}) with $\hat{\extras}-\hat{\extras}'$, then
simplify by $|\hat{\extras}-\hat{\extras}'|_{1+r}$ and we have:

\[
\begin{array}{l}
\mid \hat{\extras}-\hat{\extras}'\mid_{1+r} +\We\DT{\mid \hat{\extras}-\hat{\extras}'\mid_{1+r}}  \leq 2\varepsilon \mid \gradu-\gradu'\mid_{1+r} + \mid\Mga-\Mga'\mid_{1+r}+ \\
\hspace*{1.5cm} + C\We \left[\mid \gradu-\gradu'\mid_{1+r}(\mid\hat{\extras}\mid_{1+r}+\mid \extra\mid_{1+r}) +\mid \hat{\extras}-\hat{\extras}'\mid_{1+r} \mid \gradu' \mid_{1+r}+\right.\\
\hspace*{2.9cm}\left. +\mid \gradu' \mid_{1+r}\mid \extra-\extra'\mid_{1}\right].
\end{array}
\]
This equation writes also, thanks to the fact that
$|\extra-\extra'|_{1+r}(t)\leq |\hat{\extras}_0-\hat{\extras}'_0\mid_{1+r}$,
$|\extra|_{1+r}(t)\leq |\extras_0|_{1+r}$ and the uniform bound
(\ref{2ch48.5}):
\[
\begin{array}{rl}
(1-C \We |\gradu'|_{1+r})|\hat{\extras}-\hat{\extras}'\mid_{1+r}+\We \DT{|\hat{\extras}-\hat{\extras}'\mid_{1+r}} \leq & C(|\gradu-\gradu'|_{1+r}+|\Mga-\Mga'|_{1+r})+\\
 & + C|\gradu'|_{1+r}|\extras_0-\extras'_0\mid_{1+r}.
\end{array}
\]
and so
\[
\begin{array}{rl}
|\hat{\extras}-\hat{\extras}'\mid_{1+r}\leq & |\extras_0-\extras'_0\mid_{1+r}e^{-\frac{1}{\scriptsize \We}\int_0^t(1-C {\scriptsize \We}|\gradu'|_{1+r}){\rm d}s}+\\[1mm]
 & +C \int_0^t e^{-\frac{1}{\scriptsize \We}\int_s^t(1-C {\scriptsize \We}|\gradu'|_{1+r}){\rm d}s'}(|\gradu-\gradu'|_{1+r}(t)+\\
 &  +|\Mga-\Mga'|_{1+r}+ |\gradu'|_{1+r}(t)|\extras_0-\extras'_0\mid_{1+r}) {\rm d}s.
\end{array}
\]
Since the term in the exponential can be bounded thanks to (\ref{2ch48.25}):
\[
-\frac{1}{\We}\Int_s^t(1-C \We|\gradu'|_{1+r}){\rm d}s' \leq C \int_s^t |\gradu'|_{1+r} \leq C \sqrt{T_0} |\gradu'|_{L^2(0,T;H^{1+r})}\leq C\sqrt{T_0},
\]
we prove in $L^{\infty}(0,T;H^{1+r})$:
\begin{equation}
\label{2ch69.1}
\mid \hat{\extras}-\hat{\extras}'\mid_{1+r}(t)\leq C\left(\mid \gradu-\gradu'\mid_{L^2(0,T;H^{1})}+\mid\Mga-\Mga' \mid_{L^2(0,T;H^{1})}+\mid \extras_0 -\extras_0' \mid_{1}\right),
\end{equation}
where $C=C(\varepsilon,B, \We, S,V,\mbox{Re})$. A similar bound in
$L^2(0,T;H^{1+r})$ is easy to derive.

The estimate on the time derivative of $\hat{\extras}-\hat{\extras}'$
is similar to the one derived in (\ref{2ch43}) (still because
$(\hat{\extras}-\hat{\extras}')(t=0)=0$). By taking the $H^{1+r}(\Omega)$ norm
of the time derivative in (\ref{2ch69}), we get:
\[
\begin{array}{ll}
\left|\DP{(\hat{\extras}-\hat{\extras}')}{t}\right|_{1+r} \leq &\frac{1}{\We}|\hat{\extras}-\hat{\extras}'|_{1+r}+C \left(|\gradu-\gradu'|_{1+r}|\hat{\extras}|_{1+r}+|\gradu'|_{1+r}|\hat{\extras}-\hat{\extras}'|_{1+r}+\right.\\
 & \left. +|\gradu-\gradu'|_{1+r} |\extra|_{1+r}+|\gradu'|_{1+r}|\extra-\extra'|_{1+r}\right)+\\
 &+\frac{2\varepsilon}{\We} |\gradu-\gradu'|_{1+r}+\frac{1}{\We}|\Mga-\Mga'|_{1+r}.
\end{array}
\]
Then we take the $L^2$ in time norm of this inequality and have:
\begin{equation}
\label{2ch69.2}
\left| \DP{(\hat{\extras}-\hat{\extras}')}{t}\right|_{L^2(0,T;H^{1+r})} \leq C(|{\extras_0}-{\extras}'_0|_{1+r}+|\gradu-\gradu'|_{L^2(0,T;H^{1+r})}+|\Mga-\Mga'|_{L^2(0,T;H^{1+r})}),
\end{equation}
and finally we gather all the results in
\[
\begin{array}{c}
\mid \hat{\extras}-\hat{\extras}'\mid_{H^1(0,T;H^{1+r})}\leq C \left( \mid \gradu-\gradu'\mid_{L^2(0,T;H^{1+r})}+\mid \Mga-\Mga' \mid_{L^2(0,T;H^{1+r})}+\mid \extras_0 -\extras_0' \mid_{1+r}\right).
\end{array}
\]

Now, since $\mid e^{-\frac{t}{\mbox{ \tiny We}}}
\mid_{H^{s}(0,T)} \leq C$ for $s \geq 0$, we can state the
following properties on $\extras$ and $\extras'$:
\begin{equation}
\label{2ch69.3}
\begin{array}{rl}
 \mid \extras-\extras' \mid_{L^2(0,T;H^{1+r})} &\leq | \extra-\extra'|_{L^2(0,T;H^{1+r})}+|\hat{\extras}-\hat{\extras}'|_{L^2(0,T;H^{1+r})}\\
 & \leq C\mid \extras_0 -\extras_0' \mid_{1+r} + CT^{\epsilon'}|\hat{\extras}-\hat{\extras}'|_{H^{1}(0,T;H^{1+r})}\\
 & \leq C\mid \extras_0 -\extras_0' \mid_{1+r} +\\
 & \hspace*{5mm}+ CT^{\epsilon'}\left(\mid \gradu-\gradu'\mid_{L^2(0,T;H^{1+r})}+\mid \Mga-\Mga' \mid_{L^2(0,T;H^{1+r})}\right).
\end{array}
\end{equation}

The same can be done in $H^{r/2}(0,T;H^1)$:

\begin{equation}
\label{2ch70}
\begin{array}{rl}
 \mid \extras-\extras' \mid_{H^{\frac{r}{2}}(0,T;H^1)} &\leq | \extra-\extra'|_{H^{\frac{r}{2}}(0,T;H^1)}+|\hat{\extras}-\hat{\extras}'|_{H^{\frac{r}{2}}(0,T;H^1)}\\
 & \leq C\mid \extras_0 -\extras_0' \mid_{1} + CT^{\epsilon'}|\hat{\extras}-\hat{\extras}'|_{H^{1}(0,T;H^1)}\\
 & \leq C\mid \extras_0 -\extras_0' \mid_{1} +\\
 & \hspace*{5mm}+ CT^{\epsilon'}\left(\mid \gradu-\gradu'\mid_{L^2(0,T;H^{1+r})}+\mid \Mga-\Mga' \mid_{L^2(0,T;H^{1+r})}\right).
\end{array}
\end{equation}

In order to bound $\mid \extras-\extras'
\mid_{H^{\frac{r}{2}+\frac{1}{4}}(0,T;H^{\frac{1}{2}}(\Omega))}$, we
can do in the same way as above by lifting the initial conditions and
by replacing the scalar product in $H^1$ by the one in
$H^{\frac{1}{2}}$. We may then use Lemma \ref{GAquatredeux} and Lemma
\ref{GAcinqquatre} because the initial fields, after lifting the
initial conditions, vanish. Accurate estimates in $H^1(0,T;H^{1/2})$
enable:

\begin{equation}
\label{2ch71}
\begin{array}{l}
\mid \extras-\extras' \mid_{H^{\frac{r}{2}+\frac{1}{4}}(0,T;L^2(S_F))}\leq CT^{\epsilon'}\left(\mid \gradu-\gradu'\mid_{L^2(0,T;H^{\frac{1}{2}})}+\mid \Mga -\Mga' \mid_{L^2(0,T;H^{\frac{1}{2}})} \right)+           \\
\mbox{\hspace{5cm}} +C \mid \extras_0-\extras'_0 \mid_{\frac{1}{2}} .
\end{array}
\end{equation}

In order to bound $\mid \extras-\extras'
\mid_{H^{\frac{1+r}{2}}(0,T;L^2(\Omega))}$, we lift the non-zero
initial conditions and then process as above with the $L^2$ scalar
product. Since $T_0$ is such that $C\sqrt{T_0}\mid \gradu'
\mid_{L^2(0,T_0;H^{1+r})} < 1$ (cf. (\ref{2ch48.25}) proved in
Subsection \ref{uniform_est_P1}), for $T_0$ small enough, we have
($r<1/2<1$) :

\begin{equation}
\label{2ch73}
\mid \extras-\extras'\mid_{H^{\frac{1+r}{2}}(0,T;L^2)}\leq C \mid \extras_0-\extras'_0 \mid+ CT^{\epsilon'}\left(\mid \gradu-\gradu'\mid_{L^2(0,T;L^2)}+\mid \Mga -\Mga' \mid_{L^2(0,T;L^2)} \right).
\end{equation}

The inequalities (\ref{2ch69.3}), (\ref{2ch73}) on
$\extras-\extras'$ can be mixed into:
\begin{equation}
\label{2ch74}
\mid \extras-\extras' \mid_{K^{1+r}}\leq C\mid \extras_0-\extras'_0 \mid_{1+r} + CT^{\epsilon'}(\mid \gradu-\gradu'\mid_{K^{2+r}}+\mid \Mga -\Mga' \mid_{L^2(0,T;H^{1+r})}).
\end{equation}

Concerning $\mid \extras -\extras'
\mid_{H^{\frac{1+r}{2}}(0,T;H^{1+r})}$ we already have the estimate in
$L^2(0,T;H^{1+r})$ (see (\ref{2ch69.3})) and
$L^{\infty}(0,T;H^{1+r})$. By lifting the initial conditions, one may,
as for (\ref{2ch70}), bound the $H^1(0,T;H^{1+r})$ norm and get:
\begin{equation}
\label{2ch74.1}
\begin{array}{rl}
 \mid \extras-\extras' \mid_{H^{\frac{1+r}{2}}(0,T;H^{1+r})} \leq & C\mid \extras_0 -\extras_0' \mid_{1+r} +\\
 & +CT^{\epsilon'}\left(\mid \gradu-\gradu'\mid_{L^2(0,T;H^{1+r})}+\mid \Mga-\Mga' \mid_{L^2(0,T;H^{1+r})}\right).
\end{array}
\end{equation}

$ $ \\

We are in position now to use the results above in the intermediate
(\ref{2ch67}):
\begin{equation}
\label{2ch72}
\begin{array}{l}
\mid \vite-\vite',q-q',\phi-\phi',0\mid_{X_T^r} \leq C \mid \fsec-\fsec',a-a',0,g-g',k-k',\vite_0-\vite'_0 , 0\mid_{Y_T^r}+                     \\
\mbox{\hspace{2cm}}+C T^{\epsilon'}\mid \vite-\vite'\mid_{K^{2+r}}+C\mid \Mga -\Mga' \mid_{L^2(0,T;H^{1+r})}+C\mid \extras_0-\extras'_0 \mid_{1+r}.
\end{array}
\end{equation}
For $T_0$ sufficiently small (and the constants do not depend on
$T<T_0$), the $|\vite-\vite'|_{K^{2+r}}$ in the rhs can be absorbed by
the lhs. So we have:
\[
\mid \vite-\vite',q-q',\phi-\phi',0\mid_{X_T^r} \leq C \mid \fsec-\fsec',a-a',\Mga -\Mga',g-g',k-k',\vite_0-\vite'_0 , \extras_0-\extras'_0\mid_{Y_T^r},
\]
which may be combined with (\ref{2ch74.1}) to get the
continuity (\ref{2ch6}) of $P_1^{-1}$ in Theorem
\ref{thlineaire}. This completes the proof of this Theorem.

\end{proof}

\section{Solving the second auxiliary problem}

\label{sect4}

We want to solve
$P_2[\vite_1,\extras_1](\vite,q,\phi,\extras)=(\fsec,a,\Mga,g,k,0,0)$
(see (\ref{eq16}) for the definition of $P_2$) with $\vite_1,
\extras_1$ given and $(\vite,q,\Phi,\extras)\in X_T^*=\{
(\vite,q,\Phi,\extras)\in X_T^r / \vite(0)=0\, ,
\extras(0)=0\}$. Indeed, we prove the following well-posedness Theorem.

\begin{theorem}
\label{resolP2_simplifie}
Let $0<r<1/2$, $B>0$ be given and $(\vite_1, \extras_1)$ in a bounded subset of $K^{r+2} \times
H^{\frac{1+r}{2}}(0,T_0;H^{1+r})$. For any
$(\fsec,a,\Mga,g,k,0,0)$ in a ball $B_{Y_T^r}(0,B)$, there exists a unique
$(\vite,q,\Phi,\extras) \in X_T^r$ solution of
\begin{equation}
\label{2ch8}
P_2[\vite_1,\extras_1](\vite,q,\Phi,\extras)=(\fsec,a,\Mga,g,k,0,0).
\end{equation}
Moreover the $P_2[\vite_1,\extras_1]^{-1}$ operator is bounded
and even continuous in $B_{Y_T^r}(0,B)$
and its boundedness constant does not depend on $T<T_0$ nor
on $\vite_1,\extras_1$ in a given ball of $K^{2+r}\times H^{\frac{1+r}{2}}(0,T_0;H^{1+r})$.
\end{theorem}

\begin{proof} Incidentally, while solving the first auxiliary problem,
we found a solution of (\ref{2ch7}) which happens to have the same lhs
with some vanishing rhs. So we are going to lift the rhs of conditions
(\ref{eq16})$_2$, (\ref{eq16})$_4$, (\ref{eq16})$_5$ so as to solve
$P_2[\vite_1,\extras_1](\vite,q,\phi,\extras)=(\fsec,a,\Mga,g,k,0,0)$
which is the same as (\ref{2ch8}).

It is proved in \cite{G.AllainAMO1987} that if $(0,a,0,g,k,0,0)\in
Y_T^r$, there exists $(\vite_2,q_2,\phi_2,0)\in X_T^r$ with
continuous dependence such that:
\[
\left\{
\begin{array}{lcl}
\mbox{Re } \vite_{2,t}-(1-\varepsilon)\Delta \vite_2+\nablah q_2 &=0 & \mbox{ in } \Omega\times (0,T),\\
\mbox{div}\, \vite_2&= a & \mbox{ in } \Omega\times (0,T),\\
-q_2 \Norm +2(1-\varepsilon)\symet[\vite_2]\cdot \Norm -\alpha \partial_{\tang}(\phi_2 \Norm)&=g & \mbox{ on } S_F\times(0,T), \\
\phi_{2,t}-\left(\partial_{\tang}(\vite_2)\right)\cdot \Norm &=k & \mbox{ on } S_F\times(0,T),\\
\vite_2(t=0)&=0.&
\end{array}
\right.
\]

Instead of solving (\ref{2ch7}), we could have solved the same system
with $\vite_1$ replaced by $\vite_1+\vite_2$ and with an other rhs:
\[
P_2[\vite_1+\vite_2,\extras_1](\vite,q,\phi,\extras)=(\fsec,0,\Mga+2\varepsilon\symet[\vite_2]+\We  g_a(\gradu_2,\extras_1),0,0,0,0).
\]
Such an equation can be rewritten
\[
P_2[\vite_1,\extras_1](\vite,q,\phi,\extras)=(\fsec,0,\Mga+2\varepsilon\symet[\vite_2]+\We \left( g_a(\gradu_2,\extras)+ g_a(\gradu_2,\extras_1)\right),0,0,0,0).
\]

Let us write explicitely the Stokes part of this system:
\[
\left\{
\begin{array}{lll}
\mbox{Re }\DP{\vite}{t}-(1-\varepsilon)\Delta \vite+\nablah q-\mbox{div} \extras & = \fsec & \mbox{ in } \Omega \times (0,T)\\
\mbox{div } \vite & =0 & \mbox{ in } \Omega \times (0,T)\\
\extras\cdot \Norm -q \Norm +2(1-\varepsilon)\symet[\vite]\cdot \Norm -\alpha\partial_{\tang}(\phi \Norm) & =0 & \mbox{ on } S_F\times(0,T),\\
\phi_t-\left(\partial_{\tang}\vite\right)\cdot \Norm & =0 & \mbox{ on } S_F\times(0,T),\\
\phi(t=0)& =0 & \mbox{ in } \Omega, \\
\vite(t=0) & =0 & \mbox{ in }\Omega,\\
\vite & =0 & \mbox{ on } S_B\times(0,T),
\end{array}
\right.
\]
that can be rewritten thanks to the definition of $(\vite_2,q_2,\phi_2)$:
\begin{equation}
\label{2ch75}
\left\{
\begin{array}{lll}
\mbox{Re }\DP{(\vite_2+\vite)}{t}-(1-\varepsilon)\Delta (\vite_2+\vite)+\nablah (q_2+q)-\mbox{div} (\extras) & = \fsec & \mbox{ in } \Omega \times (0,T)\\
\mbox{div} (\vite_2+\vite) & =a & \mbox{ in } \Omega \times (0,T)\\
\extras\cdot \Norm -(q_2+q) \Norm +2(1-\varepsilon)\symet[\vite_2+\vite]\cdot \Norm -\alpha\partial_{\tang}((\phi_2+\phi) \Norm) & =g & \mbox{ on } S_F\times(0,T),\\
(\phi_2+\phi)_t-\left(\partial_{\tang}(\vite_2+\vite)\right)\cdot \Norm & =k & \mbox{ on } S_F\times(0,T),\\
(\phi_2+\phi)(t=0)& =0 & \mbox{ in } \Omega, \\
(\vite_2+\vite)(t=0) & =0 & \mbox{ in }\Omega,\\
\vite_2+\vite & =0 & \mbox{ on } S_B\times(0,T).
\end{array}
\right.
\end{equation}
Now we must write the viscoelastic part of the system
\[
\begin{array}{l}
\extras +\We \left(\DP{\extras}{t}- g_a(\gradu,\extras)- g_a(\gradu_1,\extras)- g_a(\gradu,\extras_1)\right) -2\varepsilon\symet[\vite] = \\
\hspace*{3cm}=\Mga +2\varepsilon\symet[\vite_2] +\We\left(  g_a(\gradu_2,\extras)+ g_a(\gradu_2,\extras_1)\right)\mbox{in } \Omega \times(0,T),
\end{array}
\]
that can be rewritten
\begin{equation}
\label{2ch76}
\extras+\We\left(\DP{\extras}{t}- g_a(\gradu_2+\gradu,\extras)- g_a(\gradu_1,\extras)- g_a(\gradu_2+\gradu,\extras_1)\right)-2\varepsilon\symet[\vite_2+\vite] =\Mga.
\end{equation}
Equations (\ref{2ch75}) and (\ref{2ch76}) can be rewritten
\[
P_2[\vite_1,\extras_1](\vite_2+\vite,q_2+q,\phi_2+\phi,\extras)=(\fsec,a,\Mga,g,k,0,0),
\]
which is exactly (\ref{2ch8}) in Theorem \ref{resolP2_simplifie} for $(\vite_2+\vite,q_2+q,\phi_2+\phi,\extras) \in X_T^r$.

Uniqueness of the solution and boundedness of this operator
$P_2^{-1}[\vite_1,\extras_1]$ may be proved in the same way as for the
operator $P_1^{-1}$ and even its continuity (it is nonlinear). The
proof relies on the boundedness of $\vite_1,\extras_1$ in the
appropriate spaces.

This completes the proof of Theorem \ref{resolP2_simplifie}.

\end{proof}

\section{Estimates of the error terms}

\label{resolp2}

In (\ref{pbentier}) we expand $P$ in the neighborhood of ${\boldsymbol
  \xi} \equiv 0$. If we keep only zeroth order terms of ${\boldsymbol
  \xi}$ for $P(0,0,0,0,0)$ and $P_1$, we put the ``error'' terms in
$E$. Thanks to the fact that $\phi$ is initially vanishing (which is
not the case for the other fields like the velocity), we linearize in
$\phi$, put apart zeroth order terms in $P(0,0,0,0,0)$, put the linear
(in $\phi$) term in $P_1(\vite,q,\phi,\extras)$ and higher order terms
in $E$. The operator $P_1$ is nonlinear because of our nonlinear
constitutive law. So, one may
write the full system of equations to be solved:

\begin{equation}
\label{pbentier}
\begin{array}{rcl}
P({\boldsymbol \xi},\vite,q,\phi,\extras) & = & \overbrace{P(0,0,0,0,0)}^\textrm{Order 0}+P_1(\vite,q,\phi,\extras)+ \overbrace{E({\boldsymbol \xi},\vite,q,\phi,\extras)}^{O({\boldsymbol \xi})+O(\phi^2)}            \\
                & = & (0,0,0,0,0,\vite_0,\extras_0),
\end{array}
\end{equation}
where the $E$ terms are fully written below with usual summation convention
and $\overline \xi_{ij}=\delta_{ij}+\xi_{ij}$:
\begin{equation}
\label{2ch572}
\begin{array}{rcl}
E_i^1({\boldsymbol \xi},\vite,q,\phi,\extras) & = & -(1-\varepsilon)\xi_{kj}\partial_k [\overline \xi_{lj} u_{i,l}]-(1-\varepsilon)\partial_j(\xi_{lj} u_{i,l})+\xi_{ki} \partial_k q-\sigma_{ij,k} \xi_{kj}              \\
E^2({\boldsymbol \xi},\vite,q,\phi,\extras) & = & \xi_{kj} u_{j,k}      \\
E_{ij}^3({\boldsymbol \xi},\vite,q,\phi,\extras) & = &-\We \left[ \Frac{(a-1)}{2} \left(\xi_{li} u_{k,l} \sigma_{kj}+\sigma_{ik} u_{k,l}\xi_{lj}\right) + \right.                 \\
   & & \hspace*{7mm} \left. \rule{0cm}{0.6cm} + \Frac{(a+1)}{2}\left(\sigma_{ik} \xi_{lk} u_{j,l} + u_{i,l} \xi_{lk} \sigma_{kj}\right)\right]-\varepsilon(u_{i,k} \xi_{kj}+ u_{j,k} \xi_{ki})               \\
E_i^4({\boldsymbol \xi},\vite,q,\phi,\extras) & = & -q({\cal N}_i-N_i)+(1-\varepsilon)\left[\xi_{kj} u_{i,k}+\xi_{ki} u_{j,k}\right]N_j+                  \\
  & & +(1-\varepsilon)\left[\,\overline \xi_{kj} u_{i,k}+\overline \xi_{ki} u_{j,k}\right]({\cal N}_j-N_j)+              \\
  & & +g\zeta(X_1)({\cal N}_i-N_i)+g\eta_2 {\cal N}_i-\alpha\partial_{\tang}Q_i+ \sigma_{ij}({\cal N}_j -N_j) \\
\mbox{where }      Q_1 & = & (1+h^{'2})^{\frac{-1}{2}}\left[\left(1+\Frac{2\zeta'\phi+\phi^2}{1+h^{'2}}\right)^{\frac{-1}{2}}-1+\Frac{\zeta'\phi}{1+h^{'2}}\right] \\
\mbox{and } Q_2 & = & -\zeta'\phi^2(1+h^{'2})^{\frac{-3}{2}}+(\phi+\zeta')Q_1 \\
E^5({\boldsymbol \xi},\vite,q,\phi,\extras) & = & -\partial_{\tang}(\vite)\cdot {\bf N} \left[ \left(N_2+\partial_{\tang} \eta_1\right)^{-2}-N_2^{-2}\right]-  \\
 & & \mbox{\hspace{1cm}} (N_2+\partial_{\tang} \eta_1)^{-2}\left(\partial_{\tang} u_2 \partial_{\tang} \eta_1 -\partial_{\tang} u_1 \partial_{\tang} \eta_2\right)       \\
E^6({\boldsymbol \xi},\vite,q,\phi,\extras) & = & 0 =E^7({\boldsymbol \xi},\vite,q,\phi,\extras).
\end{array}
\end{equation}

Lastly, we will lift the initial conditions with
$(\vite_1,q_1,\phi_1,\extras_1)$ in $X_{T_0}^r$. We will also use the
space of initially vanishing fields $X_T^*=\{(\vite,q,\phi,\extras)
\in X_T^r / \vite(0)=0, \extras(0)=0\}$. We will denote
$B_{X_T^*}(0,R)$ the ball of center $0$ and of radius $R$ in
$X_T^*$.

Actually we need to bound the ``error'' terms. This is done in the
following theorem.

\begin{theorem}
\label{estimationserreurs}
Let $0<r<1/2$, and $(\vite_1,q_1,\phi_1,\extras_1) \in B_{X_{T_0}^r}(0,R)$.
There
exists $\epsilon'>0$ and $0<T_0'\leq T_0$ depending on
$(\vite_1,q_1,\phi_1,\extras_1)$ and $R$, such that if
$0<T<T_0'$ and $(\vite,q,\phi,\extras) \in B_{X_{T}^*}(0,R)$, then
$E(\vite_1+\vite,q_1+q,\phi_1+\phi,\extras_1+\extras)$ is in the space
$Y_T^r(\Omega)$, and the following holds:
\begin{equation}
\label{est1}
\begin{array}{rl}
\mid E^i(\vite_1+\vite,q_1+q,\phi_1+\phi,\extras_1+\extras)\mid_{(Y_T^r)_i} & \leq C T^{\epsilon'}.
\end{array}
\end{equation}
If in addition $(\vite',q',\phi',\extras') \in B_{X_{T}^*}(0,R)$, 
the operator $E$ is contracting:
\begin{equation}
\label{est2}
\begin{array}{l}
\mid E(\vite_1+\vite,q_1+q,\phi_1+\phi,\extras_1+\extras)-E(\vite_1+\vite',q_1+q',\phi_1+\phi',\extras_1+\extras')\mid_{Y_T^r}  \\
\hspace{3cm}\leq C T^{\epsilon'} \mid \vite-\vite', q-q', \phi-\phi', \extras -\extras'\mid_{X_T^r},
\end{array}
\end{equation}
with a constant $C$ that depends on $\varepsilon, a, \We, r, R, (\vite_1,q_1,\phi_1,\extras_1)$, but not on $T$ provided $T\leq T_0'$.
\end{theorem}

We provide the proof for the first and third components more
specific to viscoelastic fluids. In the fourth component ($E^4$), we
study only the term added by viscoelasticity of the fluid. The other
components ($E^2, E^4, E^5$) are studied in \cite{G.AllainAMO1987} or
in \cite{JTBeale81}.

\subsection{Various lemmas}

We want to give a meaning to the fact that ${\boldsymbol \xi}$ is
small in small time (used in Theorem \ref{estimationserreurs}). So we
need some lemmas stated hereafter.

\begin{lemma}
\label{GAcinqdeux}
Let $0<T\leq T_0, 0 \leq s < \Frac{1}{2}$, $0\leq \epsilon' \leq s$ and
$X$ be a Hilbert space. The linear mapping
$v \mapsto V(t)=\Int_0^t v(s) \, {\rm d}s$, is a bounded
operator from $H^s(0,T;X)$ to $H^{s+1-\epsilon'}(0,T;X)$ and
\[
\mid V \mid_{s+1-\epsilon'} \leq C T^{\epsilon'} \mid v \mid_s,
\]
with $C$ independent of $0<T \leq T_0$.
\end{lemma}

\begin{proof}

The proof is through double interpolation. For $s=0$, Cauchy-Schwarz inequality gives;
\[
\mid V\mid_X (t)\leq t^{\frac{1}{2}} \mid v \mid_{H^0(0,T;X)} \; \Rightarrow \mid V \mid_0 \leq C T \mid v \mid_0,
\]
where $|.|_0=|.|_{H^0(0,T;X)}$. So $\mid V\mid_1 \leq (1+CT^2) \mid v \mid_0 \leq C\mid v \mid_0$ where $|.|_1=|.|_{H^1(0,T;X)}$.
Then, by a classical interpolation inequality ($1-\epsilon'\geq 0$),
the result is obvious for $s=0$:

\[
\mid V \mid_{1-\epsilon'} \leq C \mid V \mid_1^{1-\epsilon'} \mid V \mid_0^{\epsilon'} \leq C T^{\epsilon'} \mid v \mid_0.
\]

In the same way if $v\in H^1(0,T)$

\[
\mid V \mid_1^2 =\mid v \mid^2_0+\mid V \mid^2_0 \leq C \mid v \mid^2_0 \; \mbox{ if }T\leq T_0.
\]

Still if $v \in H^1$ : $\mid V \mid_2 \leq C \mid v \mid_1$. Moreover
if $v(0)=0$ (this assumption is discussed below), since $v$ is the
integral of $\partial_t v$, then $\mid v \mid_0 \leq C T \mid v' \mid_0 $
and $| V |_1 \leq C T | v |_1$.  The same interpolation inequality
provides:

\[
\mid V \mid_{2-\epsilon'} \leq C \mid V \mid_2^{1-\epsilon'} \mid V \mid_1^{\epsilon'}  \leq C T^{\epsilon'} \mid v \mid_1^{\epsilon'} \mid v \mid_1^{1-\epsilon'}\leq C T^{\epsilon'} \mid v \mid_1.
\]

By re-interpolating between the two inequalities, one completes the
proof.

Notice that the assumption $v(0)=0$ disappears since we restrict the
regularity to $s<1/2$. Such an assumption is meaningless for
non-regular functions. So this assumption, also done by J.T. Beale
(in his Lemma 2.4), does not limit the proof.

\end{proof}

Let us remind that we define $\mbox{}_0H^{-1}(\Omega)$ as the dual
space of $\mbox{ }^0H^1(\Omega)=\left\{p \in H^1(\Omega), \; p=0
\mbox{ on }S_F\right\}$. This space is needed for the
incompressibility condition. We may then state our next lemma.

\begin{lemma}
\label{GAcinqcinq}
Let $\Omega$ be an open subset of $\mathbb{R}^2$.
\begin{itemize}
\item[(i)] Let $r>1, r\geq s \geq 0$, $v\in H^r(\Omega)$ and $w\in H^s(\Omega)$, then $v w \in H^s(\Omega)$ and
\[
\mid v w \mid_s \leq C  \mid v \mid_r \mid w \mid_s
\]
\item[(ii)] Let $v\in H^r(\Omega)$ for $r>1$, and $w\in {}_0H^{-1}(\Omega)$, then $vw \in {}_0H^{-1}(\Omega)$ and :
\[
\mid v w \mid_{-1} \leq C  \mid v \mid_r \mid w \mid_{-1}
\]
\item[(iii)] Let $v,w \in H^1(\Omega)$, then $vw\in L^2(\Omega)$ and 
\[
\mid v w \mid_{0} \leq C  \mid v \mid_1 \mid w \mid_1
\]
\item[(iv)] Let $v\in  H^1(\Omega)$ and $w\in L^2(\Omega)$, then $vw \in {}_0H^{-1}(\Omega)$ and 
\[
\mid v w \mid_{-1} \leq C  \mid v \mid_1 \mid w \mid_0
\]
\end{itemize}
\end{lemma}

This lemma is proved in dimension 3 by J.T. Beale (Lemma 2.5 p. 366 of
\cite{JTBeale81}). The proof in 2-D is very similar. Following
J.T. Beale, we state below that the product of two functions in appropriate
spaces is a continuous map.

\begin{lemma}
\label{GAcinqsix}
Let $X,Y,Z$ denote three Hilbert spaces and $M: X\times Y \rightarrow Z$,
a bounded and bilinear map (``multiplication'').
\begin{itemize}
\item[(i)] Suppose $u\in H^s(0,T;X)$ and $v\in H^s(0,T;Y)$ where $s>1/2$ then $uv=M(u,v) \in H^s(0,T;Z)$ and  $\mid u v \mid_s\leq C \mid u\mid_s \mid v \mid_s$.
\item[(ii)] If $s\leq 2$ and $u,v$ satisfy in addition to $(i)$ the conditions $\partial_t^k u(0)=0=\partial_t^k v(0)$ for $0\leq k <s-1/2$ and $s-1/2$ is not an integer. Then the constant $C$ of $\textrm{(i)}$ does not depend on $T<T_0$.
\end{itemize}
\end{lemma}

This lemma is proved p. 366 of \cite{JTBeale81} (Lemma 2.6) in two
lines. The proof is based on an extension to $t \in \mathbb{R}$ and
estimates of the transforms of $(i)$. Lemma \ref{GAcinqquatre} enables
to prove $(ii)$.\\

We will also need the following lemma that ensures there exists a
bound independent of $T<T_0$ for fields whose initial value is not
zero.

\begin{lemma}
\label{hervecinqneuf}
Let $(\vite_1,q_1,\phi_1,\extras_1)\in X_{T_0}^r$ ($r<1/2$) and
$(\vite,q,\phi,\extras)$ in a ball of radius $R$ in $X_T^r$ with
$\vite(0)=0, \; \extras(0)=0$ (so in $X_T^*$). Then 
the following holds with constants $C$ that depend on $R, \vite_1,q_1,\phi_1,\extras_1, r,
\Omega$, but not on $T<T_0$:
\begin{eqnarray}
\label{5.9.1}
\mid \partial_k \partial_j (\vite_1+\vite)\mid_{K^r(0,T;\Omega)}\leq C \\
\label{5.9.2}
\mid \partial_k (\vite_1+\vite)\mid_{K^{r+1}(0,T;\Omega)}\leq C \\
\label{5.9.3}
\mid \nablah (q_1+q)\mid_{K^r(0,T;\Omega)} \leq C                 \\
\label{5.9.4}
\mid (\extras_1+\extras)_{ij,k} \mid_{K^r(0,T;\Omega)} \leq C
\end{eqnarray}
\end{lemma}

The proof relies on two arguments. On the one hand
$(\vite_1,q_1,\phi_1,\extras_1)$ does not need to be extended until
$T_0$.  On the other hand $\vite,q,\phi,\extras$ are either initially
vanishing (for $\vite$ and $\extras$) and so can be extended on
$(0,T_0)$ with a bound independent of $T<T_0$, or not very regular
($p$ and $\phi$) and so the extension on $(0,T_0)$ does not
generate any cost nor need any extra assumption. At that level, we see
that improving to $r > 1/2$ our result would modify the whole proof.

\begin{proof}

Concerning (\ref{5.9.1}), Lemma 2.1 of \cite{JTBeale81} (which deals
with the boundedness of the derivative and trace operators) and our
Lemma \ref{GAquatredeux} (which
deals with the extension operator from $(0,T)$ to $(0,+\infty)$)
enable to state:

\[
\begin{array}{rcl}
\mid \partial_k \partial_j (\vite_1+\vite)\mid_{K^r(0,T;\Omega)} & \leq & \mid \partial_k \partial_j \vite_1\mid_{K^r(0,T_0;\Omega)}+ \mid \partial_k \partial_j\vite\mid_{K^r(0,T;\Omega)}\\
 & \leq &C \mid \vite_1\mid_{K^{r+2}(0,T_0;\Omega)}+C \mid \vite\mid_{L^2(0,T;H^{2+r})\bigcap H^{r/2}(0,T;H^2)},\\
 & \leq &C \mid \vite_1\mid_{K^{r+2}(0,T_0;\Omega)}+C \mid \vite\mid_{K^{r+2}(0,T_0;\Omega)} ,\\
 & \leq & C.
\end{array}
\]

One may prove in the same way (\ref{5.9.2}).

Concerning the pressure, since $r<1/2$, we do not need to force
unphysical vanishing initial value to have the result thanks to Lemma
2.1 of \cite{JTBeale81}.

The extra stress bound can be written:

\[
\mid (\extras_1+\extras)_{,j}\mid_{K^r}\leq \mid \extras_{1 \, ,j}\mid_{K^r(0,T_0)}+\mid\extras_{,j}\mid_{K^r(0,T)}\leq C(T_0)\mid \extras_1\mid_{K^{r+1}(0,T_0)}+C\mid \extras\mid_{K^{r+1}(0,T_0)}\leq C.
\]

\end{proof}

The following lemma gives a sense to the claim that ``${\boldsymbol \xi}$ is small''.

\begin{lemma}
\label{GAcinqsept}
Let $r$ be such that $0<r<1/2$, and let us denote ${\cal A}$ the algebra
\begin{equation}
\label{def_A}
{\cal A}=H^{1+\frac{r}{8}}(0,T;H^{1+\frac{r}{2}}(\Omega)).
\end{equation}
For any $\vite_1 \in K^{2+r}(0,T_0), \vite, \vite' \in K^{2+r}(0,T)$, $\vite(t=0)=0=\vite'(t=0)$, and ${\boldsymbol \xi}$ defined by ${\boldsymbol \xi}(\vite) =({\bf I}+\textrm{d} {\boldsymbol \eta} )^{-1}-{\bf I}=({\bf I}+\int_0^t \gradu )^{-1}-{\bf I}$, there exists $\epsilon'>0$ and constants $C$ such that if $T<T_0$;
\begin{eqnarray}
\label{est2.5}
\mid {\boldsymbol \xi}(\vite_1+\vite) \mid_{H^{\frac{1+r}{2}}(0,T;H^{1+r/2})} \leq \mid {\boldsymbol \xi}(\vite_1+\vite) \mid_{\cal A} &\leq &C T^{\epsilon'}\\
\label{est2.625}
\mid {\boldsymbol \xi}(\vite_1+\vite)-{\boldsymbol \xi}(\vite_1+\vite')\mid_{H^{\frac{1+r}{2}}(0,T;H^{1+r/2})} \leq \mid {\boldsymbol \xi}(\vite_1+\vite)-{\boldsymbol \xi}(\vite_1+\vite')\mid_{\cal A} & \leq & C T^{\epsilon'}\mid \vite-\vite'\mid_{K^{r+2}},
\end{eqnarray}
with $C$, dependent on $\vite_1, r, T_0$ but not on $T$ provided $T< T_0$.
\end{lemma}

The following Remark must be taken into account before we prove this lemma. 

\begin{remark} 
\label{remarque2}
The continuity bound of the product in the algebra ${\cal A}$ depends
on $T<T_0$ if there is no more condition.  If we add the condition
that the fields are vanishing initially, Lemma \ref{GAcinqsix}
(proved with an extension operator whose properties are given
in Lemma \ref{GAquatredeux}) may
ensure that the bound does not depend on $T<T_0$, but what happens
if this assumption is not satisfied ?

To answer this question, let us write
the continuity of the product of two functions that do not depend on
time and so do not vanish at $t=0$:
\[
\mid 1 \times 1 \mid_{H^{1+\frac{r}{8}}(0,T)} \leq C({\cal A}) \mid 1 \mid_{H^{1+\frac{r}{8}}(0,T)}\mid 1 \mid_{H^{1+\frac{r}{8}}(0,T)} \Rightarrow 1 \leq C({\cal A}) \mid 1 \mid_{H^{1+\frac{r}{8}}(0,T)}.
\]
So the constant not only depends on $T$, but even tends to $+\infty$
when $T\rightarrow 0$.

One might wonder whether the Lemma \ref{GAquatredeux}, needed in the
proof of Lemma \ref{GAcinqsix}, uses a too specific extension which could be improved. Of
course, the extension used will not work for instance on constant
functions which will not remain in any $H^s(0,\infty)$. But the above
inequality proves that no other extension operator could suit. This
explains why in Remark \ref{remarque1}, we claimed we were forced to
lift initial conditions so as to have new initially vanishing fields.
\end{remark}

\begin{proof}

We denote ${\boldsymbol \xi}:={\boldsymbol \xi}(\vite_1+\vite)$, ${\boldsymbol \xi}':={\boldsymbol \xi}(\vite_1+\vite')$ and
${\boldsymbol \eta}, {\boldsymbol \eta}'$ the associated fields respectively. We repeatedly use
the fact that ${\cal A}$ is an algebra and that $\vite$ is in $K^{2+r}$
and so also in $H^{\frac{1+r}{2}}(0,T;H^1)$. So ${\boldsymbol \xi} \in H^1(0,T;H^{1+r})
\bigcap H^{\frac{3+r}{2}}(0,T;L^2)$.

Let us denote ${\cal E}$ the extension to $(0,+\infty)$ operator and
${\cal R}$ the restriction to $(0,T)$ operator. From Lemma
\ref{GAquatredeux} we know that ${\cal E}$ is bounded independently of
$T<T_0$ for convenient fields. So:

\[
\begin{array}{rl}
\mid  \textrm{d}{\boldsymbol \eta}  \textrm{d}{\boldsymbol \eta}' \mid_{H^{1+\frac{r}{8}}(0,T;H^{1+r/2})}  & = \mid {\cal R} {\cal E} \textrm{d}{\boldsymbol \eta} {\cal R} {\cal E} \textrm{d}{\boldsymbol \eta}' \mid_{H^{1+\frac{r}{8}}(0,T;H^{1+r/2})} \\
 & \leq C \mid {\cal R}{\cal E} \textrm{d}{\boldsymbol \eta} \mid_{H^{1+\frac{r}{8}}(0,T;H^{1+r/2})} \mid {\cal R}{\cal E} \textrm{d}{\boldsymbol \eta}' \mid_{H^{1+\frac{r}{8}}(0,T;H^{1+r/2})} \\
  & \leq C \mid {\cal E} \textrm{d}{\boldsymbol \eta} \mid_{H^{1+\frac{r}{8}}(\mathbb{R};H^{1+r/2})} \mid {\cal E} \textrm{d}{\boldsymbol \eta}' \mid_{H^{1+\frac{r}{8}}(\mathbb{R};H^{1+r/2})}  \\
  & \leq  C \mid \textrm{d}{\boldsymbol \eta} \mid_{\cal A}\mid \textrm{d}{\boldsymbol \eta}' \mid_{\cal A},
\end{array}
\]

where $C$ does not depend on $T < T_0$. As a consequence there exists
$\alpha_1$ such that if $\mid \textrm{d}{\boldsymbol \eta} \mid_{\cal A}<\alpha_1$,
then, ${\boldsymbol \xi} =({\bf I}+\textrm{d} {\boldsymbol \eta}
)^{-1}-{\bf I}$ is in ${\cal A}$ and it can be expanded in series
with:

\begin{equation}
\label{est2.75}
\mid {\boldsymbol \xi} \mid_{\cal A} \leq C \mid  \textrm{d} {\boldsymbol \eta} \mid_{\cal A}.
\end{equation}

Since $\textrm{d}{\boldsymbol \eta}(X,t)=\Int_0^t \nablah (\vite_1+\vite)$, $\vite_1 \in
K^{r+2}(0,T_0)$ and the fact that $\vite$ is only in $K^{r+2}(0,T)$ but with
vanishing initial condition, Lemmas \ref{hervecinqneuf} and
\ref{GAcinqdeux} enable to state the announced result (\ref{est2.5}):

\[
\mid {\boldsymbol \xi} \mid_{\cal A} \leq C \mid \textrm{d} {\boldsymbol \eta} \mid_{H^{1+r/8}(0,T;H^{1+r/2})} \leq C T^{\epsilon'} \mid \nablah(\vite_1+\vite) \mid_{H^{\frac{r}{16}}(0,T;H^{1+r/2})} \leq C T^{\epsilon'}.
\]
Moreover, since ${\cal A}$ is an algebra, if $\mid \textrm{d} {\boldsymbol \eta} \mid
< \alpha_1/2$ and the same for $\textrm{d} {\boldsymbol \eta}'$;
\[
{\boldsymbol \xi} -{\boldsymbol \xi}' = ({\bf I}+{\boldsymbol \xi}) \textrm{d}({\boldsymbol \eta}'-{\boldsymbol \eta}) ({\bf I}+{\boldsymbol \xi}') =\textrm{d}({\boldsymbol \eta} -{\boldsymbol \eta}')+{\boldsymbol \xi} \textrm{d}({\boldsymbol \eta}-{\boldsymbol \eta}')+\textrm{d}({\boldsymbol \eta}-{\boldsymbol \eta}') {\boldsymbol \xi}' +{\boldsymbol \xi} \textrm{d}({\boldsymbol \eta}-{\boldsymbol \eta}') {\boldsymbol \xi}'.
\]
Here, all the functions are initially vanishing. So, in the same way
as for (\ref{est2.75});
\begin{equation}
\label{2.875}
\mid  {\boldsymbol \xi}- {\boldsymbol \xi}' \mid_{\cal A} \leq C \mid \textrm{d} ({\boldsymbol \eta}-{\boldsymbol \eta}')\mid_{\cal A}.
\end{equation}

One may then prove (\ref{est2.625}) in the same way as (\ref{est2.5}).
\end{proof}

The previous lemma will not be sufficient in some estimates. So we state
the following lemma which will turn to be useful.

\begin{lemma}
\label{lemma5.6.5}
If the velocity $\vite$ is in $L^2(0,T;H^{2+r})$, then the field
\[
{\boldsymbol \xi}=({\bf Id}+{\rm d} {\boldsymbol \eta})^{-1}-{\bf Id}= ({\bf Id}+\Int_0^t \nablah \vite )^{-1}-{\bf Id}
\]
is in $H^1(0,T;H^{1+r})$ with
\[
\mid  {\boldsymbol \xi} \mid_{H^1(0,T;H^{1+r})} \leq C.
\]
Moreover, for $0<r<1$, we define the algebra
\begin{equation}
\label{def_Aprime}
{\cal A}'=H^{\frac{1+r}{2}}(0,T;H^{1+r}(\Omega)),
\end{equation}
for which
\[
\mid {\boldsymbol \xi} \mid_{{\cal A}'} \leq C T^{\epsilon'}.
\]
\end{lemma}

The proof is similar to the one of Lemma \ref{GAcinqsept} for the
first estimate. Then one must use Lemma \ref{GAcinqtrois} to bound
$\mid {\boldsymbol \xi} \mid_{{\cal A}'}$ with a $C T^{\epsilon'}$ and
$\mid {\boldsymbol \xi} \mid_{H^1(0,T;H^{1+r})} $. The rest is left to
the reader.

In the proof of estimates in $L^2(0,T;H^s)$ below, we will use
$L^{\infty}_t$ estimates and then $H^{\frac{1+r}{2}}(0,T) \hookrightarrow
L^{\infty}(0,T)$. The constants in Sobolev's inequality must be
independent of $T<T_0$. In the general case, it is wrong but for the
subspace of initially vanishing functions, it is true as states the
following lemma.

\begin{lemma}
\label{lemma57}
Let $0 < r \leq 1$. If $v \in H^{\frac{1+r}{2}}(0,T)$ and is initially vanishing then
\begin{equation}
\label{eq57}
\mid v \mid_{L^{\infty}(0,T)} \leq C \mid v \mid_{H^{\frac{1+r}{2}}(0,T)},
\end{equation}
and the constant $C$ does not depend on $T < T_0$.
\end{lemma}

\begin{proof}
The proof of $H^1(\mathbb{R}) \hookrightarrow
L^{\infty}(\mathbb{R})$ is classical. Then, thanks to the properties of the
extension operator ${\cal E}$ in $H^1$ (Lemma \ref{GAquatredeux} $(ii)$),

\[
\mid v \mid_{L^{\infty}(0,T)}\leq \mid {\cal E}(v) \mid_{L^{\infty}(\mathbb{R})} \leq \mid {\cal E}(v) \mid_{H^1(\mathbb{R})} \leq C \mid v \mid_{H^1(0,T)},
\]

with $C$ independent of $T <T_0$.

The same can be proved for non-integer Sobolev spaces and completes the proof.
\end{proof}

For the estimates in $H^{\frac{r}{2}}(0,T;L^2)$, we will need a more
precise result than Lemma \ref{GAcinqsix} to estimate the product of
two functions.

\begin{lemma}
\label{lemma58}
Let $X,Y,Z$ denote three Hilbert spaces and $M: X\times Y \rightarrow Z$,
a bounded and bilinear map (``multiplication'' ). Let $1/2 < s < 3/2$ and $0\leq s' \leq s$.
\begin{itemize}
\item[(i)] Suppose $u\in H^s(0,T;X)$ and $v\in H^{s'}(0,T;Y)$ then $uv:=M(u,v) \in H^{s'}(0,T;Z)$ and  $\mid u v \mid_{s'}\leq C \mid u\mid_s \mid v \mid_{s'}$.
\item[(ii)] Let $u,v$ satisfy the conditions of $(i)$ and in addition $u(t=0)=0$ and $s'<1/2$. Then there exists a constant $C$ in $\textrm{(i)}$ that does not depend on $T<T_0$.
\end{itemize}
\end{lemma}

\begin{proof}
Let us consider the functional $v \mapsto uv$. For $u \in
H^s(0,T;X)$, this function can be defined in $H^0(0,T;Y)$ and in
$H^s(0,T;Y)$. It satisfies:
\[
\begin{array}{rl}
\mid uv \mid_{H^0(0,T;Z)} \leq & C \mid u \mid_{H^{s}(0,T;X)} \mid v \mid_{H^0(0,T;Y)}, \\
\mid uv \mid_{H^{s}(0,T;Z)} \leq & C \mid u \mid_{H^{s}(0,T;X)} \mid v \mid_{H^{s}(0,T;Y)}.
\end{array}
\]
Then a simple interpolation provides the result ($i$) for $s'$ between $0$ and $s$.\\

Then the extension operator enables to exhibit $C$ independent of $T<T_0$:
\[
\mid uv \mid_{H^{s'}(0,T)} \leq \mid {\mathcal E}(u) {\mathcal E}(v) \mid_{H^{s'}(\mathbb{R})} \leq C \mid {\mathcal E}(u) \mid_{H^{s}(\mathbb{R})} \mid {\mathcal E}(v) \mid_{H^{s'}(\mathbb{R})} \leq C \mid u \mid_{H^{s}(0,T)} \mid v \mid_{H^{s'}(0,T)}.
\]

\end{proof}

\subsection{Estimates on $E^1$}

The terms to be estimated are in (\ref{2ch572}).

For the Navier-Stokes equations, G. Allain \cite{G.AllainAMO1987}
sends back to \cite{G.Allaintoulouse} where she indicates the tools to
get accurate estimates and refers to J.T. Beale's article
\cite{JTBeale81} for details.

We denote $\xi_{kj}(\vite_1+\vite)=\xi_{kj}=\left( ({\bf
 Id}+\textrm{d} {\boldsymbol \eta} )^{-1}-{\bf Id}\right)_{kj}$
because it depends on the velocity $\vite_1+\vite$ through
${\boldsymbol \eta}$. For the sake of completeness we consider in
detail the first term of $E^1$ :
$\xi_{kj}(\vite_1+\vite)\partial_k [\overline\xi_{lj}(\vite_1+\vite)
\times(\vite_1+\vite)_{i,l}]$. We prove it is bounded in $K^r(0,T)$
by $C T^{\epsilon'}$ (see (\ref{est1})) and contracting (cf
(\ref{est2})).

The main difficulty here is that the constants for the embedding
($H^{\frac{1+r}{2}}(0,T) \hookrightarrow L^{\infty}(0,T)$) or the
property of algebra ($\mid f g\mid_{H^s} \leq C \mid f
\mid_{H^{\frac{1+r}{2}}} \mid g \mid_{H^s}$ for $0 \leq s <1/2$) tend
to infinity when $T$ tends to $0$ in the general case.
Yet, since the fields involved are initially vanishing, we can use our
more refined lemmas proved above.

Moreover, for the reader not familiar with the $K^r$ spaces, we prefer to split the estimates in $L^2(0,T;H^r)$ and $H^{\frac{r}{2}}(0,T;L^2(\Omega))$.

Below, all the constants $C$ are independent of $T <T_0$.

\underline{In $L^2(0,T;H^r)$.}\\

Here, we can use the algebra $L^{\infty}(0,T;H^{1+r})$. Then, denoting $ {\boldsymbol \xi}= {\boldsymbol \xi}(\vite_1+\vite)$,
\[
\begin{array}{lll}
\mid \xi_{kj}\partial_k [\oxi_{lj} \times (\vite_1+\vite)_{i,l}] \mid_{L^2(0,T;H^r)} & \leq & C \mid \xi_{kj} \mid_{L^{\infty}(0,T;H^{1+r})} \mid \partial_k [\overline \xi_{lj}\times (\vite_1+\vite)_{i,l}] \mid_{L^2(0,T;H^r)} \\
\hspace*{4cm} & \leq & C \mid \xi_{kj} \mid_{H^{\frac{1+r}{2}}(0,T;H^{1+r})} \mid \overline \xi_{lj} \times(\vite_1+\vite)_{i,l} \mid_{L^2(0,T;H^{1+r})},
\end{array}
\]
thanks to Lemma \ref{lemma57}. Then, one may use Lemma
\ref{lemma5.6.5} and write ${\boldsymbol \oxi} ={\bf Id} +
{\boldsymbol \xi}$ to pursue the bound:
\[
\begin{array}{lll}
\hspace*{4cm} & \leq & C T^{\epsilon'} \left( \mid (\vite_1+\vite)_{i,l})\mid_{L^2(0,T;H^{1+r})}+\mid \xi_{lj}\times (\vite_1+\vite)_{i,l})\mid_{L^2(0,T;H^{1+r})} \right).
\end{array}
\]

Then one has, with arguments similar to above and with $\xi_{lj}(t=0)=0$:

\[
\begin{array}{lll}
\hspace*{3cm} & \leq & C T^{\epsilon'} \left( \mid \vite_1+\vite \mid_{L^2(0,T;H^{2+r})}+C T^{\epsilon'} \mid \vite_1+\vite \mid_{L^2(0,T;H^{2+r})} \right)\\
\hspace*{3cm} & \leq & C T^{\epsilon'},
\end{array}
\]

where $C$ does depend on $R$ but not on $T < T_0$.

\underline{In $H^{\frac{r}{2}}(0,T;L^2)$.}\\

We can no more take off the ${\boldsymbol \oxi}$ term by a simple $L^{\infty}$
bound. So we need algebra properties (in time) and so Lemma
\ref{lemma58} which is a more refined version of Lemma
\ref{GAcinqsix}. Indeed the latter is useless in
$H^{\frac{r}{2}}(0,T;L^2)$ since $r/2 < 1/2$ and so $H^{r/2}$ is not
an algebra.

Because of Lemma \ref{lemma58} and $\mathcal{A}
=H^{1+\frac{r}{8}}(0,T;H^{1+\frac{r}{2}})$,

\[
\begin{array}{lll}
\mid \xi_{kj}\partial_k [\overline \xi_{lj}\times (\vite_1+\vite)_{i,l}] \mid_{H^{\frac{r}{2}}(0,T;L^2)} & \leq & C \mid \xi_{kj} \mid_{\mathcal{A}} \mid \partial_k [\overline \xi_{lj}\times (\vite_1+\vite)_{i,l}] \mid_{H^{r/2}(0,T;L^2)} \\
\hspace*{4cm} & \leq & C T^{\epsilon'} \mid \overline \xi_{lj}\times (\vite_1+\vite)_{i,l} \mid_{H^{r/2}(0,T;H^{1})}.
\end{array}
\]

The last inequality uses Lemma \ref{GAcinqsept}. Then a simple
decomposition of ${\boldsymbol \oxi}={\bf Id} +{\boldsymbol \xi}$ enables to have a
$(\vite_1+\vite)_{i,l}$ term not initially vanishing but linear, and a
${\boldsymbol \xi}(\vite_1+\vite) (\vite_1+\vite)_{i,l}$ term which is nonlinear but
initially vanishing:

\[
\begin{array}{lll}
\hspace*{4cm} & \leq & C T^{\epsilon'} \left( \mid \vite_1+\vite \mid_{H^{r/2}(0,T;H^2)}+C T^{\epsilon'} \mid \vite_1+\vite \mid_{H^{r/2}(0,T;H^2)} \right).
\end{array}
\]

The difference with the estimate in $L^2(0,T;H^r)$ is that the
equivalent term was $\mid \vite_1+\vite\mid_{L^2(0,T;H^{2+r})}$. Such
a term could be bounded by $\mid \vite_1+\vite \mid_{K^{2+r}(0,T)}$ with
a constant equal to 1. But here we need to use Lemma
\ref{GAcinqquatre} and especially its {\em (ii)} because we
need constants independent of $T< T_0$. So we split $\vite_1 +\vite$,
use the fact that $\vite_1 \in K^{2+r}(0,T_0)$ (costless) and that
$\vite(t=0)=0$, so that $\vite \in K^{2+r}(0,T)$ can be extended
without any loss to $K^{2+r}(0,T_0)$. In summary, we have proved the
following bound:

\[
\begin{array}{lcl}
\mid \xi_{kj}(\vite_1+\vite)\partial_k [\overline \xi_{lj}(\vite_1+\vite)\times (\vite_1+\vite)_{i,l}] \mid_{H^{\frac{r}{2}}(0,T;L^2)} & \leq & C T^{\epsilon'}.
\end{array}
\]

The other terms are no more difficult. So (\ref{est1}) is established for $E^1$.

\begin{remark}
The argument that ${\boldsymbol \eta}(\vite_1+\vite)=\int_0^t (\vite_1+\vite)$ is
initially vanishing and so that the embedding of $H^{\frac{1+r}{2}}$
in $L^{\infty}$ has a constant independent of $T < T_0$ is not in
p. 380 of \cite{JTBeale81}. We only added this and Lemmas
\ref{lemma57} and \ref{lemma58} to complete the proof.
\end{remark}

One could prove in the same way that $E^1$ is lipschitz and so
satisfies (\ref{est2}).

\subsection{Estimates on $E^3$}

We consider now the error terms of the constitutive equation in
$K^{r+1}((0,T)\times \Omega)$ :

\[
\begin{array}{c}
\mid E^3_{ij}({\boldsymbol \xi}(\vite_1+\vite),\vite_1+\vite,q_1+q,\phi_1+\phi,\extras_1+\extras)\mid_{K^{r+1}} \leq \\
\hspace{2cm} C \mid \xi_{li}(\vite_1+\vite) \times (\vite_1+\vite)_{k,l} (\extras_1+\extras)_{kj}\mid_{K^{r+1}}+C \mid(\vite_1+\vite)_{i,k}\xi_{kj}(\vite_1+\vite)\mid_{K^{r+1}}.
\end{array}
\]

In order to treat the first term in $K^{1+r}$, we split $K^{1+r}$ in
$L^2(0,T;H^{1+r})$ and $H^{\frac{1+r}{2}}(0,T;L^2(\Omega))$. We apply
Lemma \ref{hervecinqneuf}, Lemma \ref{lemma5.6.5} and Lemma
\ref{lemma57} to the first term in $L^2(0,T;H^{1+r})$:

\[
\begin{array}{l}
\mid\xi_{li}(\vite_1+\vite) (\vite_1+\vite)_{k,l} (\extras_1+\extras)_{kj}\mid_{L^2(0,T;H^{1+r})}  \\
\hspace{2cm} \leq \mid \xi_{li}(\vite_1+\vite)\mid_{\cal A'}\mid(\vite_1+\vite)_{k,l}\mid_{L^2(0,T;H^{1+r})}\mid (\extras_1+\extras)_{kj}\mid_{L^{\infty}(0,T;H^{1+r})}   \\
\hspace{2cm} \leq C T^{\epsilon'},
\end{array}
\]

where $\cal A'$ is defined in (\ref{def_Aprime}). One may also bound
the first term in $H^{\frac{1+r}{2}}(0,T;L^2)$. Indeed, Lemma
\ref{GAcinqsept} and Lemma \ref{lemma58} enable to write:

\[
\begin{array}{l}
\mid\xi_{li}(\vite_1+\vite) (\vite_1+\vite)_{k,l} (\extras_1+\extras)_{kj}\mid_{H^{\frac{1+r}{2}}(0,T;L^2)} \\
\hspace*{2cm} \leq C \mid \xi_{li}(\vite_1+\vite)\mid_{\cal A'} \times                     \\
\hspace{2.3cm} \times (\mid (\vite_1)_{k,l}\mid_{H^{\frac{1+r}{2}}(0,T_0;L^2)}+\mid\vite_{k,l}\mid_{H^{\frac{1+r}{2}}(0,T;L^2)})\mid \extras_1+\extras\mid_{H^{\frac{1+r}{2}}(0,T;H^{1+r})}\\
\hspace{2cm} \leq C T^{\epsilon'}(\mid \vite_1,q_1,\phi_1,\extras_1\mid_{X_{T_0}^r} + C R)^2\\
\hspace{2cm} \leq C T^{\epsilon'},
\end{array}
\]

with $C$ depending on $\We, a,\vite_1, q_1, \phi_1, \extras_1, T_0$,
but not on $T\leq T_0$. So we proved:

\[
\mid \xi_{li}(\vite_1+\vite) (\vite_1+\vite)_{k,l} (\extras_1+\extras)_{kj}\mid_{K^{r+1}} \leq C T^{\epsilon'} .
\]

In order to treat the second term, the algebra properties of
${\cal A'}=H^{\frac{1+r}{2}}(0,T;H^{1+r})$ give:

\[
\mid (\vite_1+\vite)_{i,k}\xi_{kj}(\vite_1+\vite)\mid_{K^{r+1}} \leq C T^{\epsilon'}\mid \vite_1+\vite \mid_{K^{r+2}} \leq C T^{\epsilon'},
\]

thanks to Lemma \ref{hervecinqneuf}. \\

In order to prove the contracting property, we compute the difference:

\[
\begin{array}{c}
\mid E^3_{ij}({\boldsymbol \xi}(\vite_1+\vite),\vite_1+\vite,q_1+q,\phi_1+\phi,\extras_1+\extras)-   \\
\hspace{1cm}  E_{ij}^3({\boldsymbol \xi}(\vite_1+\vite'),\vite_1+\vite',q_1+q',\phi_1+\phi',\extras_1+\extras')\mid_{K^{r+1}}  \leq \\
\hspace{2cm} 4 C \mid (\xi_{li}(\vite_1+\vite)-\xi_{li}(\vite_1+\vite'))(\vite_1+\vite)_{k,l}(\extras_1+\extras)_{kj} \mid_{K^{r+1}}+ \\
\hspace{2cm}+4C\mid \xi_{li}(\vite_1+\vite)(\vite-\vite')_{k,l}(\extras_1+\extras)_{kj} \mid_{K^{r+1}}+ \\
\hspace{2cm}+4C\mid \xi_{li}(\vite_1+\vite)(\vite_1+\vite')_{k,l}(\extras-\extras')_{kj} \mid_{K^{r+1}}+ \\
\hspace{2cm}+2C \mid (\vite-\vite')_{i,k} \xi_{kj}(\vite_1+\vite)\mid_{K^{r+1}}+                         \\
\hspace{2cm}+2C \mid (\vite_1+\vite')_{i,k}(\xi_{kj}(\vite_1+\vite)-\xi_{kj}(\vite_1+\vite'))\mid_{K^{r+1}}.
\end{array}
\]

Every term can be managed in the same way as before by using 
(\ref{est2.625}) and not (\ref{est2.5}) of Lemma \ref{GAcinqsept}, Lemma \ref{lemma5.6.5}, Lemma \ref{lemma57} and Lemma \ref{lemma58}.

\subsection{Estimates on $E^4$}

This operator comes almost only from the Navier-Stokes part. This part
has been estimated by G. Allain \cite{G.Allaintoulouse} for a
Newtonian fluid. So we only have to estimate the terms
$(\extras_1+\extras)_{ij}({\cal N}_j-N_j)$ in
$K^{r+\frac{1}{2}}(S_F\times (0,T))$. We will estimate separately
$\extras_1+\extras$ and ${\cal N}_j-N_j$.

Concerning $\extras_1+\extras$, as for Lemma \ref{hervecinqneuf}, we have:

\[
\begin{array}{rcl}
\mid \extras_1+\extras\mid_{K^{r+\frac{1}{2}}}(S_F\times(0,T)) & \leq & \mid \extras_1\mid_{K^{r+\frac{1}{2}}(S_F\times(0,T))}+\mid \extras\mid_{K^{r+\frac{1}{2}}(S_F\times(0,T))}        \\
     & \leq & \mid \extras_1\mid_{K^{r+\frac{1}{2}}(S_F\times(0,T_0))}+C \mid \extras\mid_{K^{r+\frac{1}{2}}(S_F\times(0,T_0))}        \\
     & \leq & \mid \vite_1,q_1,\phi_1,\extras_1\mid_{X_{T_0}^r}+C \mid \vite,q,\phi,\extras\mid_{X_{T_0}^r}      \\
     & \leq & \mid \vite_1,q_1,\phi_1,\extras_1\mid_{X_{T_0}^r}+CR,
\end{array}
\]
thanks to Lemma \ref{GAcinqquatre} $(ii)$. So $\extras_1+\extras$ is bounded in
$K^{r+\frac{1}{2}}(S_F \times (0,T))$ by a constant independent of $T<T_0$.
Concerning ${\cal N}_j-N_j$, we use the following formula:
\[
{\boldsymbol{\mathcal{N}}}-{\bf N}=\Int_0^t \left(-\partial_{\tang}((\vite_1)_2+(\vite)_2)(s),\partial_{\tang}((\vite_1)_1+(\vite)_1)(s)\right)\; {\rm d}s.
\]
The estimate of ${\cal N}_j-N_j$ is done in the algebra
$H^{1+\frac{r}{8}}(0,T;H^{\frac{1}{2}+\frac{r}{2}}(S_F))$ so as to
conclude. Since $\partial_{\tang}(\vite_1+\vite)\in
L^2(0,T;H^{\frac{1}{2}+\frac{r}{2}}(S_F))$, Lemmas \ref{GAcinqdeux} and
\ref{hervecinqneuf} apply and also classical theorems found in
\cite{LionsMagenes}:

\[
\begin{array}{l}
\mid \Int_0^t \left( -\partial_{\tang}((\vite_1)_2+(\vite)_2)(s),\partial_{\tang}((\vite_1)_1+(\vite)_1)(s)\right)ds \mid_{H^{1+\frac{r}{8}}(0,T;H^{\frac{1}{2}+\frac{r}{2}}(S_F))}     \\
\hspace{3cm} \leq C T^{\epsilon'}\mid \left( -\partial_{\tang}((\vite_1)_2+(\vite)_2),\partial_{\tang}((\vite_1)_1+(\vite)_1)\right)\mid_{L^2(0,T;H^{\frac{1}{2}+\frac{r}{2}}(S_F))}   \\
\hspace{3cm} \leq C T^{\epsilon'}\mid \nablah(\vite_1+\vite)\mid_{L^2(0,T;H^{1+\frac{r}{2}})}             \\
\hspace{3cm} \leq C T^{\epsilon'}.
\end{array}
\]

Thanks to the fact that $1+\Frac{r}{8}> \Frac{1}{2}>
\Frac{1}{4}+\Frac{r}{2}$, the term satisfies:

\[
\begin{array}{rl}
\mid (\extras_1+\extras)_{ij}({\cal N}_j-N_j)\mid_{K^{r+\frac{1}{2}}(S_F\times (0,T))}\leq & C\mid \extras_1+\extras \mid_{K^{r+1/2}(S_F)} \mid {\cal N}_j-N_j \mid_{H^{1+\frac{r}{8}}(0,T;H^{\frac{1}{2}+\frac{r}{2}}(S_F))}  \\
\leq & C T^{\epsilon'} \mid \extras_1+\extras \mid_{K^{r+\frac{1}{2}}(S_F)}\leq C T^{\epsilon'}.
\end{array}
\]

The contracting property of this operator is proved in the same way.\\

This completes the proof of Theorem \ref{estimationserreurs} 

\section{Fixed point}

We want to solve the Lagrangian nonlinear system (\ref{eq14})
associated to the operator $P$. Let us remind the reader with the
expansion done in the first Section:

\begin{equation}
\label{2chP1}
\begin{array}{rcl}
P({\boldsymbol \xi},\vite,q,\phi,\extras) & = & P(0,0,0,0,0)+P_1(\vite,q,\phi,\extras)+E({\boldsymbol \xi},\vite,q,\phi,\extras)            \\
                & = & (0,0,0,0,0,\vite_0,\extras_0)
\end{array}
\end{equation}

where $P(0,0,0,0,0)=(0,0,0,g_0 \zeta(X_1) N_i-\alpha\partial_{\tang}
(\tang),0,0,0)$ contains all the zeroth order terms (gravity and
initial surface tension).\\

First we lift the initial conditions and the zeroth order terms. To
that end, we use Theorem \ref{thlineaire} that states that $P_1$ is
invertible from $X_T^r$ to $Y_T^r$:
$P_1(\vite,q,\phi,\extras)=(\fsec,a,\Mga,g,k,\vite_0,\extras_0)$ with
continuous dependence on the initial conditions. So let
$(\vite_1,q_1,\phi_1,\extras_1)\in X_{T_0}^r$ be such that:
\[
P_1(\vite_1,q_1,\phi_1,\extras_1)=(0,0,0,0,0,\vite_0,\extras_0)-P(0,0,0,0,0),
\]
which is allowed thanks to the fact that the right-hand side is in
$Y_{T_0}^r$ ($\zeta \in H^{\frac{5}{2}+r}$ implies that 
$\partial_{\tang}(\tang) \in H^{r+\frac{1}{2}}$).

If we perform a change of variable for the unknown fields
($\vite_1+\vite$ replaces $\vite$ and so on), we are led to find
$(\vite,q,\phi,\extras)\in X_T^*=\{(\vite,q,\phi,\extras) \in X_T^r /
\vite(0)=0, \extras(0)=0\}$ such that:

\[
\begin{array}{c}
P_1(\vite_1+\vite,q_1+q,\phi_1+\phi,\extras_1+\extras)+E({\boldsymbol \xi}(\vite_1+\vite),\vite_1+\vite,q_1+q,\phi_1+\phi,\extras_1+\extras)=         \\
\hspace*{1cm}  \rule{6cm}{0cm} P_1(\vite_1,q_1,\phi_1,\extras_1) \\
\Leftrightarrow P_2[\vite_1,\extras_1](\vite,q,\phi,\extras)=-E({\boldsymbol \xi}(\vite_1+\vite),\vite_1+\vite,q_1+q,\phi_1+\phi,\extras_1+\extras),
\end{array}
\]

where $P_2$ is the second auxiliary problem introduced above in
(\ref{eq16}). Since $P_2[\vite_1,\extras_1]$ is invertible when the
rhs is in $Y_T^r$ thanks to Theorem \ref{resolP2_simplifie}, we want
to solve

\[
\begin{array}{rcl}
(\vite,q,\phi,\extras) & = & P_2^{-1}[\vite_1,\extras_1](-E({\boldsymbol \xi}(\vite_1+\vite),\vite_1+\vite,q_1+q,\phi_1+\phi,\extras_1+\extras)) \\
 & := & F(\vite,q,\phi,\extras).
\end{array}
\]

At that level, G. Allain proves in \cite{G.Allaintoulouse} for a Newtonian fluid that 
her $F$ lets an invariant ball. But since we use a contraction mapping,
we do not need this.

%
%

We need now to prove that $F$ is contracting. Thanks to
Lemma \ref{GAcinqsept} and especially (\ref{est2}):

\[
\mid F(\vite,q,\phi,\extras)-F(\vite',q',\phi',\extras')\mid_{X_T^*}\leq C T^{\epsilon'}\mid \vite-\vite', q-q', \phi-\phi', \extras -\extras'\mid_{X_T^*},
\]

where $C$ depends on the parameters, and also on the functions that
lift the intial conditions, but only through a bound of their norm. As
before, it does not depend on $T<T_0$, so we have proved that for $T$
sufficiently small, $F$ is contracting. We may apply the contraction
mapping principle which provides a solution to the Lagrangian system
of equations (\ref{eq14}). Owing to the regularity of the Lagrangian
velocity $\vite$, one has solved also the Eulerian equations
(\ref{eq7}) by a simple change of variables.

Thanks to the contraction property, uniqueness is obvious.

Last the solution depends on the initial conditions directly in a
continuous way and indirectly through the lift functions also
continuously.

The proof of our main Theorem \ref{theor2} is complete.

\appendix

\section{Appendix}

The proof of our main theorem is written for the specific constitutive
law of Oldroyd. In order to convince the reader that it can include most
reasonnable constitutive laws, we study successively various models
well-known in the literature. We could even treat a new model
of an elastoviscoplastic fluid \cite{Saramito_07}, but it would
require deeper modifications of the proof.

The most crucial step in the proof is the first one: when we prove
uniform (in $n$ and $T_0$) estimate on $\extras^n$. Then, using this
estimate, it is easy to derive the same inequalities as ours until the
conclusion.

\subsection{The Giesekus constitutive law}

In \cite{Giesekus}, Giesekus provides a new constitutive law:

\[
\extras + \We \displaystyle \frac{{\cal D}_1[{\bf v}] }{{\cal D} t} \extras+ c_{Giesekus} \extras^2 = 2 \varepsilon \symet[\vite],
\]

where $\We, \varepsilon, c_{Giesekus}$ are positive parameters. We
propose to discretize a reduced version of it in which we keep only
the new terms:

\[
\extras^{n+1}+\We \DP{\extras^{n+1}}{t}+c_{Giesekus} \extras^n \,\extras^{n+1} =2 \varepsilon  \symet[\vite^n].
\]

If we take the $H^{1+r}$ scalar product of this equation with $\extras^{n+1}$, one has

\[
| \extras^{n+1} |_{1+r}^2 +\Frac{\We}{2} \Frac{{\rm d}}{{\rm d} t} (|\extras^{n+1} |_{1+r}^2) \leq (2\varepsilon | \gradu^n |_{1+r} + c_{Giesekus} |\extras^n|_{1+r})|\extras^{n+1} |_{1+r}.
\]

The same computations as in our proof leads to

\[
\begin{array}{l}
\mid \extras^{n+1}\mid_{1+r}^2(t) \leq                                \\
C \Int_0^t (e^{-2\Int_s^t\Frac{(1/2-c_{Giesekus} \We \mid \extras^n \mid_{1+r})}{\We}{\rm d}t'}\mid \gradu^n \mid_{1+r}^2){\rm d}s.
\end{array}
\]

And as above, we may bound the term in the exponential:

\[
\begin{array}{rl}
-2\Int_s^t\Frac{(1/2-c_{Giesekus} \We \mid \extras^n \mid_{1+r})}{\We}{\rm d}t' & \leq \Frac{2 c_{Giesekus}}{\We} \int_s^t |\extras^n|_{1+r} {\rm d} t'\\
 & \leq \Frac{2c_{Giesekus} S T_0}{\We}.
\end{array}
\]

So if we require that $T_0$ be such that $2c_{Giesekus} S T_0 / \We <1$, we have

\[
|\extras^{n+1}|_{1+r}(t) \leq C (| \gradu^n |_{L^2(0,T;H^{1+r})}) \leq C V.
\]

It suffice then to require $CV \leq S$ to get the uniform in $n$
estimate on $|\extras^{n+1}|_{1+r}$.

\subsection{The Phan-Thien Tanner constitutive law}

In \cite{Phan-Thien_Tanner_77}, Phan-Thien and Tanner derive new
models of viscoelasticity from a molecular argument using a network
which is allowed to be non-affine:

\[
Y_{\varepsilon_{PTT}}({\rm tr} \extras) \extras + \We \displaystyle \frac{{\cal D}_a[{\bf v}] }{{\cal D} t} \extras = 2 \varepsilon \symet[\vite],
\]
where $Y_{\varepsilon_{PTT}}(x)=\exp{(\varepsilon_{PTT} \We \, x)}$ in the exponential model and
$Y(x)=1+\varepsilon_{PTT} \We \, x$ in the linear model.

We propose to discretize the exponential law in which already
treated terms are removed:
\[
\extras^{n+1} +\We \DP{\extras^{n+1}}{t}=2\varepsilon \symet[\vite^n] +(1-e^{\varepsilon_{PTT} {\rm tr} \extras^n})\extras^{n+1}.
\]

We take the $H^{1+r}$ scalar product of this equation with
$\extras^{n+1}$ and make the same computation as above to have:

\[
\begin{array}{l}
\mid \extras^{n+1}\mid_{1+r}(t) \leq                                \\
C \Int_0^t (e^{-\Int_{t'}^{t}\Frac{(1-\mid 1-e^{\varepsilon_{PTT} {\rm tr} \extras^n} \mid_{1+r})}{\We}{\rm d}t''}(\mid \gradu^n \mid_{1+r})){\rm d}t'.
\end{array}
\]
One is led to estimate the integral:
\[
\begin{array}{rcl}
-\Int_{t'}^{t}\Frac{(1-\mid 1-e^{\varepsilon_{PTT} {\rm tr} \extras^n} \mid_{1+r})}{\We}{\rm d}t'' & \leq & C\int_{t'}^t \mid 1-e^{\varepsilon_{PTT} {\rm tr} \extras^n} \mid_{1+r} {\rm d}t''\\
 & \leq & C\int_{t'}^t \left| \sum_{k=1}^{+\infty} \varepsilon_{PTT}^k ({\rm tr}\extras^n)^k \right|_{1+r} {\rm d} t''\\
 & \leq & C(e^{\varepsilon_{PTT} \, S}-1)T_0.
\end{array}
\]

So a condition on the PTT parameter $\varepsilon_{PTT}$ enables to ensure the uniform in
$n$ (and in $T_0$) estimate on $|\extras^{n+1}|_{1+r}$.

\end{document}